%% file: main-Sep-30.tex
\documentclass[11pt,reqno]{amsart}

\usepackage{fullpage}

\usepackage{graphicx}
\usepackage{tikz-cd}
\usepackage{float}
\usepackage{subcaption}

\usepackage[colorlinks=true, pdfstartview=FitH, linkcolor=blue, citecolor=blue, urlcolor=blue]{hyperref}

\usepackage{amssymb,amsmath,amsthm, url, comment}

\usepackage{enumitem}

\begin{document}
\newtheorem{thm}{Theorem}[section]
\newtheorem{prop}[thm]{Proposition}
\newtheorem{lem}[thm]{Lemma}

\newtheorem{cor}[thm]{Corollary}

\theoremstyle{definition}
\newtheorem{defn}[thm]{Definition}
\newtheorem{notations}[thm]{Notation}
\newtheorem{remark}[thm]{Remark}
\newtheorem{remarks}[thm]{Remarks}
\newtheorem{example}[thm]{Example}
\newtheorem{question}[thm]{Question}
\newtheorem{reduction}[thm]{Reduction}
\newtheorem{conjecture}[thm]{Conjecture}
\newtheorem*{ack}{Acknowledgment}

\newtheoremstyle{claimstyle}
{}{} 
{\itshape} 
{} 
{\itshape} 
{.} 
{ } 
{}

\theoremstyle{claimstyle}
\newtheorem{claim}[thm]{Claim}
\newtheorem*{claim-star}{Claim}

\renewcommand{\labelenumi}{(\alph{enumi})}

\newcommand{\C}{{\mathbb C}}
\newcommand{\Q}{{\mathbb Q}}
\newcommand{\N}{{\mathbb N}}
\newcommand{\R}{{\mathbb R}}
\newcommand{\Z}{{\mathbb Z}}
\newcommand{\rk}{\mathrm{rk}}
\renewcommand{\P}{\mathbb{P}}
\newcommand{\A}{\mathbb{A}}
\renewcommand{\bar}{\overline}
\renewcommand{\tilde}{\widetilde}
\newcommand{\Gal}{\mathrm{Gal}}
\newcommand{\bQ}{\bar{\Q}}
\newcommand{\eps}{\varepsilon}

\title{A formula for the rank over $\mathbb{Q}(t)$ of the elliptic curve $y^2=x^3+At^6+Bt^3+C$}

\author{Zhengheng Bao}
\address{40 St. George Street, Toronto ON, Canada, M5S 2E4}
\email{zhengheng.bao@mail.utoronto.ca}

\begin{abstract}

In this paper, we give an explicit formula for the rank and generators (up to finite index) over $\mathbb{Q}(t)$ of all non-trivial elliptic curves of the form $y^2=x^3+At^6+Bt^3+C$, which is a larger class of elliptic surfaces than the one in \cite{Kl} and \cite{DN}, namely $y^2=x^3+At^6+C$. Our proof provides a new method to find this formula using generators of the geometric Mordell-Weil group while this geometric Mordell-Weil group (which is a $\mathbb{Z}[\omega]$-module) no longer decomposes into rank-one submodules by the methods in \cite{Kl} and \cite{DN}. Moreover, our proof uses only a few elements in the Galois group of the coordinates of the generators and only light computations by hand.

\end{abstract}

\maketitle

\section{Introduction}

Let $f(t) \in \Q[t]$ and $E:y^2=x^3+f(t)$ be an elliptic curve. For a field $\Q \subseteq k \subseteq \bQ$, elliptic curves over $k(t)$ correspond to elliptic surfaces over $k$ (\cite{SS} chapter 5). We use these terms interchangeably. 

All elliptic surfaces in this paper are nontrivial, meaning that $f(t)$ does not equal $ct^{6m}$ for some $c \in \C$ and $m \in \Z$. 
Some research on the rank and generators of the geometric Mordell-Weil group $E(\bQ(t))$ and of the arithmetic Mordell-Weil group $E(\Q(t))$ for some specific $f(t)$ has been done so far. 

On the geometric Mordell-Weil group, the rank for $1 \leq \deg(f) \leq 6$ was given by Shioda-Tate's formula (\cite{SS} Section 7.2) early in the history. The highest rank known so far is $68$, given by Shioda \cite{Shi92} in 1992 with a proof that the rank for $f(t)=t^m+1$ is at most $68$ with equality if and only if $360 \mid m$.
A method to write $E(\bQ(t))$ as a direct sum of Mordell-Weil groups of lower-rank surfaces (up to finite index) was given by Chahal, Meijer, and Top \cite{CMT} in 2000 and is used in our proof. 
The rank for $f(t)=t^m+1$ for every $m \geq 1$ was given by Usui \cite{Usu} in 2008. Explicit generators for $f(t)=t^m+1$ for each $1 \leq m \leq 12$ were given by Salami and Zargar \cite{SZ} in 2025. Explicit generators for $f(t)=t^{360}+1$ were given by Salami \cite{Sal} in 2025.

On the explicit formulas for the arithmetic Mordell-Weil group, 
a result on the rank and generators for $1 \leq \deg(f) \leq 3$ was given by Bremner \cite{Br} in 1991. 
A result on the rank of elliptic surfaces in the general form but still with low-degree coefficients was given by Battistoni, Bettin, and Delaunay \cite{BBD} in 2021 for elliptic surfaces of the form $y^2=x^3+\alpha_2(t)x^2+\alpha_4(t)x+\alpha_6(t)$, where $\deg(\alpha_2), \deg(\alpha_4) \leq 2$, $\deg(\alpha_6) \leq 3$, $\alpha_6$ is monic, and one of $\alpha_2,\alpha_4$ is nonzero. 
A result on the rank and generators for higher-degree $f(t)$ but of a specific form $f(t)=At^6+B$ was first given by Desjardins and Naskrecki \cite{DN} in 2024. 
A different and shorter proof of \cite{DN} was then given by Kloosterman \cite{Kl} in 2026, which inspired the proof of this paper.

On the algorithmic aspects for the arithmetic Mordell-Weil group, Butler and Elsenhans \cite{BE} in 2026 found an algorithm which, for any given rational elliptic surface, computes the Galois representation (where $\Gal(\bQ/\Q)$ acts on the coordinates) of the geometric Mordell-Weil group, which automatically implies the arithmetic rank. 
However, their algorithm does not provide a closed-form formula in terms of the coefficients, like the $A,B,C$ in our paper. 
It computes the Galois group, chooses a good prime $p$ that depends on the Galois group, uses mod $p$ reduction for this $p$, and uses $l$-adic techniques.
Moreover, their algorithm needs to 
apply all elements in the Galois group to the generating set of the geometric Mordell-Weil group, 
and then compute the height pairing matrix after each Galois action is applied. 
This is sometimes time-consuming even on the computer because of the large size of the Galois group and the complexity of each computation.

In this paper, we give an explicit formula for the arithmetic rank and generators (up to finite index) for $y^2=x^3+At^6+Bt^3+C$, which is a larger class of elliptic surfaces than the one in \cite{Kl} and \cite{DN}. 
Our proof provides a new method to find this formula using the geometric generators while this geometric Mordell-Weil group does not decompose into rank-one submodules by the methods in \cite{Kl} and \cite{DN}. Moreover, our proof uses only a few elements in the Galois group of the generators and only light computations by hand. Additionally, our formula also leads to a non-trivial upper bound on the arithmetic rank for these elliptic surfaces.

Now we fix some notations for elliptic surfaces that we are going to use frequently. 

\begin{notations} \label{notations-E}
    Let $A,B,C \in \Q$, and $(A,B),(B,C) \neq (0,0)$. 
    We denote 
    $$E_0: y^2=x^3+(At^6+Bt^3+C),$$
    $$E_1: y^2=x^3+(At^2+Bt+C), \ E_2: y^2=x^3+(At^4+Bt^3+Ct^2), \ E_3: y^2=x^3+(At^6+Bt^5+Ct^4),$$
    and for $f(t) \in \Q(t)$, denote 
    $$E_f: y^2=x^3+f(t).$$
    Note that the last condition we put on $A,B,C$ is equivalent to $E_0$ being nontrivial.
\end{notations}

\begin{notations}
    We denote the third root of unity by $\omega=e^{\frac{2\pi i}{3}}$, and the complex multiplication $(x,y) \mapsto (\omega x,y)$ on elliptic curves of the form $E_f$ also by $\omega$, which will not cause ambiguity. 
    Note that some Mordell-Weil groups later in this paper will be $\Z[\omega]$-modules, so we add a subscript $\Z$ in the rank notation for abelian groups such as $\rk_{\Z}E_f(\Q(t))$.
\end{notations}

The main theorems of this paper are a formula for $\rk_{\Z}E_0(\Q(t))$ in Theorem~\ref{result-(6,3,0)} and explicit generators (up to finite index) of $E_0(\Q(t))$ in Theorem~\ref{arith-gen-(6,3,0)} that follow from the proof of Theorem~\ref{result-(6,3,0)}. 

\begin{thm} \label{result-(6,3,0)}
    We have:
    \begin{enumerate}
    
    \item 
    $$\rk_{\Z} E_0(\Q(t)) = \rk_{\Z} E_1(\Q(t))+\rk_{\Z} E_2(\Q(t))+\rk_{\Z} E_3(\Q(t));$$

    \item 
    $$\rk_{\Z} E_1(\Q(t)) = \begin{cases}
        1 \quad \text{if $\sqrt{A} \in \Q(\omega)\backslash\{0\}$ and $\sqrt[3]{\frac{B^2-4AC}{4A}} \in \Q(\omega)\backslash\{0\}$;}\\
        0 \quad \text{otherwise.}
    \end{cases}$$

    \item 
    If $A=0$, then 
    $$\rk_{\Z} E_2(\Q(t)) = \begin{cases}
         1 \quad \text{if $\sqrt[3]{B} \in \Q(\omega)\backslash\{0\}$ and $\sqrt{C} \in \Q(\omega)\backslash\{0\}$};\\
         0 \quad \text{otherwise}.
    \end{cases}$$
    
    If $A\neq 0, C=0$, then 
    $$\rk_{\Z} E_2(\Q(t)) = \begin{cases}
        1 \quad \text{if $\sqrt[3]{B} \in \Q(\omega)\backslash\{0\}$ and $\sqrt{A} \in \Q(\omega)\backslash\{0\}$}; \\
        0 \quad \text{otherwise}.
    \end{cases}$$

    If $A,C \neq 0, B^2-4AC=0$, then 
    $$\rk_{\Z} E_2(\Q(t)) = \begin{cases}
        1 \quad \text{if $\sqrt[3]{2B} \in \Q(\omega)\backslash\{0\}$ and $\sqrt{A} \in \Q(\omega)\backslash\{0\}$};\\
        0 \quad \text{otherwise}.
    \end{cases}$$

    If $A,C,B^2-4AC \neq 0$, then 
    \begin{gather*}\rk_{\Z} E_2(\Q(t))= \\ 
    \begin{cases}
        2 \quad \text{if $\sqrt{A},\sqrt{C} \in \Q(\omega)$ and
    $\sqrt[3]{2 \sqrt{A}\sqrt{C}-B}, \sqrt[3]{-2 \sqrt{A}\sqrt{C}-B} \in \Q(\omega)$;}\\
        1 \quad \text{if (at least) one of the following holds:}\\
        \ \quad \text{\ \ \ \ (i) Exactly one of $\sqrt{A},\sqrt{C}$ lies in $\Q(\omega)$,}\\
        \ \quad \text{\ \ \ \ \ \ \ \ and both of $\sqrt[3]{2\sqrt{A}\sqrt{C}-B},\sqrt[3]{-2\sqrt{A}\sqrt{C}-B}$ lie in $\Q(\omega,\sqrt{AC})$};\\
        \ \quad \text{\ \ \ \ (ii) Both of $\sqrt{A},\sqrt{C}$ lie in $\Q(\omega)$,} \\
        \ \quad \text{\ \ \ \ \ \ \ \ and exactly one of $\sqrt[3]{2\sqrt{A}\sqrt{C}-B},\sqrt[3]{-2\sqrt{A}\sqrt{C}-B}$ lies in $\Q(\omega)$;}\\
        0 \quad \text{otherwise}.
    \end{cases}\end{gather*}

    \item 
    $$\rk_{\Z} E_3(\Q(t)) = \begin{cases}
        1 \quad \text{if $\sqrt{C} \in \Q(\omega)\backslash\{0\}$ and $\sqrt[3]{\frac{B^2-4AC}{4C}} \in \Q(\omega)\backslash\{0\}$;}\\
        0 \quad \text{otherwise.}
    \end{cases}$$

    \end{enumerate}

\end{thm}

\begin{thm} \label{arith-gen-(6,3,0)}
    We have:
    \begin{enumerate}
        \item 
        For $i \in \{1,2,3\}$, let $S_i$ be a basis of a finite index subgroup of $E_i(\Q(t))$. 
        Then the set 
        $$ S_0=\left\{\left(\frac{x(t^3)}{t^{2(i-1)}}, \frac{y(t^3)}{t^{3(i-1)}}\right) \bigg| \ i \in \{1,2,3\}, (x(t),y(t))\in S_i\right\}$$
        is a basis of a finite index subgroup of $E_0(\Q(t))$.
        
        \item 
        Suppose $\rk_{\Z} E_1(\Q(t))>0$. Let 
        $P=\left( \sqrt[3]{\frac{B^2-4AC}{4A}}, \sqrt{A}t +\frac{B}{2\sqrt{A}}   \right)$ where we take the real cube root. 
        Then a basis (up to finite index) of the group $E_1(\Q(t))$ is 
        $\begin{cases}
            \{P\} \quad \text{if $\sqrt{A} \in \Q$}\\
            \{P+2\omega P\} \ \text{otherwise}.
        \end{cases}$
        
        \item 
        Suppose that $\rk_{\Z}E_2(\Q(t))>0$ and either $A=0$, $C=0$, or $B^2-4AC=0$. 
        Let 
        $$P=\begin{cases}
            (-\sqrt[3]{B} \cdot t, \sqrt{C}\cdot t) \quad \text{if $A=0$;}\\
            (-\sqrt[3]{B} \cdot t, \sqrt{A}\cdot t^2) \quad \text{if $A \neq 0$, $C=0$;}\\
            \left( -\sqrt[3]{\frac{B^4}{256A^3}} \bigg( -\frac{16A^2}{B^2} \left( t+\frac{B}{4A} \right)^2 +1  \right),\\ 
            \ \ \ \sqrt{\frac{B^2}{16A}} \left(  -\frac{16A^2}{B^2} \left( t+\frac{B}{4A} \right)^3 + \left( t+\frac{B}{4A} \right) \right)  \bigg) \quad \text{if $A,C \neq 0, B^2-4AC=0$,}
        \end{cases}$$ 
        where we take the real cube root in all three cases.
        
        Then a basis (up to finite index) for the group $E_2(\Q(t))$ is $\begin{cases}
            \{P\} \quad \text{if $\sqrt{A}, \sqrt{C} \in \Q$}\\
            \{P+2\omega P\} \quad \text{otherwise}.
        \end{cases}$
        
        \item 
        Suppose that $\rk_{\Z}E_2(\Q(t))>0$ and $A, C, B^2-4AC \neq 0$. 
        Let 
        $$\ \ \ \ \ \ \ \ P=\left(\sqrt[3]{2 \sqrt{A}\sqrt{C}-B} \cdot t, \   \sqrt{A} t^2 + \sqrt{C} t\right), \ Q=\left(\sqrt[3]{-2 \sqrt{A}\sqrt{C}-B} \cdot t, \   \sqrt{A} t^2 - \sqrt{C} t\right).$$
        
        (i) If $\rk_{\Z}E_2(\Q(t))=2$, then $\sqrt{A},\sqrt{C} \in \Q(\omega)$, and the two cube roots lie in $\Q(\omega)$. If the cube root can be real, we take the real cube root. If not, we can take the cube roots so that the two cube roots are conjugate in $\Q(\omega)/\Q$. 
        
        Then a basis (up to finite index) for the group $E_2(\Q(t))$ is 
        $$\begin{cases}
            \{(1+2\omega)P, (1+2\omega)Q\} \quad &\text{if $\sqrt{A}, \sqrt{C} \notin \Q$}\\
            \{P-Q, (1+2\omega)(P+Q)\} \quad &\text{if $\sqrt{A} \notin \Q$, $\sqrt{C} \in \Q$}\\
            \{P+Q, (1+2\omega)(P-Q)\} \quad & \text{if $\sqrt{A} \in \Q$, $\sqrt{C} \notin \Q$}\\
            \{P,Q\} \quad & \text{if $\sqrt{A}, \sqrt{C} \in \Q$}.
        \end{cases}$$

        (ii) If $\rk_{\Z}E_2(\Q(t))=1$ and exactly one of $\sqrt{A},\sqrt{C}$ lies in $\Q(\omega)$, then the two cube roots lie in $\Q(\omega,\sqrt{AC})$, and we can take the two cube roots so that they are conjugate in $\Q(\omega,\sqrt{AC})/\Q(\omega)$. 
        We let 
        $$R=\begin{cases}
            P-Q \quad \text{if $\sqrt{A}\notin \Q(\omega)$, $\sqrt{C} \in \Q(\omega)$;}\\
            P+Q \quad \text{if $\sqrt{A}\in \Q(\omega)$, $\sqrt{C} \notin \Q(\omega)$}.
        \end{cases}$$ 
        Let $\tau$ be the nontrivial element in $\Gal(\Q(\omega)/\Q)$. 
        Then a basis (up to finite index) for the group $E_2(\Q(t))$ is 
        $$\begin{cases}
            \{(1+2\omega)R\} \quad \text{if $\tau(R)=-R$;}\\
            \{R+\tau(R)\} \quad \text{otherwise}.
        \end{cases}$$

        (iii) If $\rk_{\Z}E_2(\Q(t))=1$ and both of $\sqrt{A},\sqrt{C}$ lie in $\Q(\omega)$, then exactly one of $P,Q$ lies in $E_2(\Q(\omega)(t))$. 
        
        Let $R$ be the one of $P,Q$ that lies in $E_2(\Q(\omega)(t))$. 
        
        Then a basis (up to finite index) for the group $E_2(\Q(t))$ is 
        $$\begin{cases}
            \{(1+2\omega)R\} \quad \text{if $\tau(R)=-R$;}\\
            \{R+\tau(R)\} \quad \text{otherwise}.
        \end{cases}$$
        
        \item 
        Suppose $\rk_{\Z} E_3(\Q(t))>0$. Let 
        $P=\left( \sqrt[3]{\frac{B^2-4AC}{4C}}\cdot t^2, \  \frac{B}{2\sqrt{C}} \cdot t^3 + \sqrt{C}\cdot t^2  \right),$ where we take the real cube root.
        
        Then a basis (up to finite index) of the group $E_3(\Q(t))$ is 
        $\begin{cases}
            \{P\} \quad \text{if $\sqrt{C} \in \Q$}\\
            \{P+2 \omega P\} \quad \text{otherwise}.
        \end{cases}$
    \end{enumerate}
\end{thm}

Now we give an outline of our paper.   
In Section \ref{section-decomposition}, we cite a result from \cite{CMT} to show that we only have to work with the groups $E_i(\Q(\omega)(t))$ ($i \in \{0,1,2,3\}$) which are $\Z[\omega]$-modules by the complex multiplication, and we use 
another result in \cite{CMT} to decompose $E_0(\Q(\omega)(t))$ as a direct sum of $E_1(\Q(\omega)(t))$, $E_2(\Q(\omega)(t))$, and $E_3(\Q(\omega)(t))$ up to finite index. 
In Section \ref{section-Bremner}, we cite some results in \cite{Br} for low-degree $f(t)$ 
and rearrange them into the form we need. 
In Section \ref{section-(6,3,0)}, we prove Theorem~\ref{result-(6,3,0)} with the main focus on part (c) which finds $\rk_{\Z} E_2(\Q(t))$. 
This is the part where the geometric Mordell-Weil group $E_2(\bQ(t))$ has $\Z[\omega]$-rank $2$ and cannot be decomposed into rank-one submodules by the methods in \cite{Kl} and \cite{DN}. 
In Section~\ref{section-arith-gen}, we prove Theorem~\ref{arith-gen-(6,3,0)} which finds the explicit arithmetic generators for Theorem~\ref{result-(6,3,0)}. 
In Section~\ref{section-remarks}, we find the largest possible rank of $E_0(\Q(t))$, give examples to achieve this rank in different ways, and give an alternative proof for the parts in \cite{Br} that we have used. 

The proof of Theorem \ref{result-(6,3,0)} part (c) in Section ~\ref{section-(6,3,0)} can be outlined as follows. 
In Lemma ~\ref{rank-criteria} and Lemma~\ref{G-acts-on-M}, we prove that $\rk_{\Z[\omega]}E_2(\Q(\omega)(t))$ can be determined by the linear relations between a basis (up to finite index) and their Galois conjugates for varying $\sigma \in G:=\Gal(\bQ/\Q(\omega))$. These lemmas imply \cite{Kl} Lemma 2.8. In Proposition ~\ref{gen-over-C}, we find independent generators $P,Q$ for our rank-$2$ module $E_2(\bQ(t))$. In Theorem~\ref{result-(4,3,2)}, 
we 
investigate the relations between $P-P^{\sigma}$ and $Q-Q^{\sigma}$ for a few $\sigma \in G$ and prove Theorem~\ref{result-(6,3,0)} part (c).

\section*{Acknowledgement}

I would like to express my sincere gratitude to my supervisor Julie Desjardins for her invaluable supervision and support throughout this work, and for her valuable feedback on all the earlier drafts of this paper. 
I would like to thank the University of Toronto for their financial support for my PhD program, during which this research was conducted. I would also like to thank Sajad Salami for his helpful comments on the first arXiv version of this paper.

\section{Decomposition} \label{section-decomposition}

In this section, we first cite a result in \cite{CMT} to show that we only have to work with $E_i(\Q(\omega)(t))$, which are $\Z[\omega]$-modules by the complex multiplication $\omega: (x,y) \mapsto (\omega x,y)$. 
Then we use another result in \cite{CMT} to decompose $E_0(\Q(\omega)(t))$ as a direct sum of $E_1(\Q(\omega)(t))$, $E_2(\Q(\omega)(t))$, and $E_3(\Q(\omega)(t))$ up to finite index. We also recall some simple change of variables that will be used later in this paper.

\begin{lem} \label{omega} (\cite{CMT} Lemma 2.2)
    Let $K$ be a field that does not contain $\sqrt{-3}$ and $L=K(\sqrt{-3})$. 
    Let $E$ be an elliptic curve of the form $y^2=x^3+b$ where $b \in K$. If $\rk_{\Z} E(K)$ is finite, then $\rk_{\Z} E(L)=2\rk_{\Z} E(K)$.
\end{lem}

\begin{notations}
    Let $f(t) \in \Q(t)$. We denote $R(f(t))= \rk_{\Z} E_f(\Q(\omega)(t))$ and $r(f(t))=\rk_{\Z} E_f(\Q(t))$. Then Lemma ~\ref{omega} tells us that $R(f(t))=2r(f(t))$.
\end{notations}

Lemma~\ref{CMT-2.1} and Lemma~\ref{reduction} decompose $E_0(\Q(\omega)(t))$ as a direct sum of $E_i(\Q(\omega)(t))$ ($i \in \{1,2,3\}$) up to finite index. 
Lemma~\ref{CMT-2.1} and part (c) of Lemma~\ref{reduction} are proved exactly in \cite{CMT} Lemma 2.1, while Lemma~\ref{reduction} parts (a) and (b) are stated only for the ground field being $\bQ(t)$ instead of $\Q(\omega)(t)$ in the middle of the proof of \cite{CMT} Proposition 4.1.
They also mentioned, without proof, that adapting their proof of Lemma 2.1 in \cite{CMT} gives Lemma~\ref{reduction} parts (a) and (b), but we prefer not to reproduce the entire proof.
Therefore, we give a simple proof to derive parts (a) and (b) from part (c) of Lemma ~\ref{reduction}.

\begin{lem} \label{CMT-2.1} (\cite{CMT} Lemma 2.1)
    Let $k/\Q(\omega)$ be a field extension and let $K=k(s)$ be a purely transcendental extension of $k$. Consider $L:=k(t)$ over $K$, in which $t^6=s$. Then $L/K$ is cyclic of degree $6$. We fix a generator $\sigma \in \Gal(L/K)$ such that $\sigma(t)=-\omega^2 t$. 
    Let $E_{f(s)}$ be the elliptic curve defined by the equation $y^2=x^3+f(s)$ for some rational function $f(s) \in K$.
    
    Then 
    $V:=E_{f(s)}(L)\otimes_{\Z}\Q$ is a $\Q(\omega)$-vector space and  
    $\sigma$ acts $\Q(\omega)$-linearly on $V$, 
    and the $\Q(\omega)$-vector space $V$ splits as a direct sum 
    $$\sum_{n=0}^5 V_{(-\omega^2)^n},$$
    where $V_{\lambda}$ is the eigenspace of $\sigma$ corresponding to the eigenvalue $\lambda$.
    Further, the $\Q(\omega)$-vector space $V_{(-\omega^2)^n}$ can be identified with $E_{s^nf(s)}(K) \otimes_{\Z} \Q$. 
    
\end{lem}

Note that we will also use the identification in their proof of the above lemma, which sends $(x(s),y(s)) \in E_{s^nf(s)}(K)$ to $(x(t^6)/t^{2n}, y(t^6)/t^{3n})\in E_{f(t^6)}(L)=E_{f(s)}(L)$. We further identify $E_{s^nf(s)}(K)$ with $E_{t^nf(t)}(L)$ naturally.

\begin{lem} \label{reduction}
    Let $f(t) \in \Q(t)$. We have:
    
    (a) $R(f(t^3))=R(f(t))+R(t^2f(t))+R(t^4f(t))$;
    
    (b) $R(f(t^2))=R(f(t))+R(t^3f(t))$;
    
    (c)(\cite{CMT} Lemma 2.1) $R(f(t^6))=R(f(t))+R(tf(t))+R(t^2f(t))
    +R(t^3f(t))+R(t^4f(t))+R(t^5f(t))$.
\end{lem}

\begin{proof}

    (c) This is stated exactly in \cite{CMT} Lemma 2.1.

    (a)(b) Let $L=\Q(\omega)(t)$, $K=\Q(\omega)(t^6)$, and fix a generator $\sigma$ of $\Gal(L/K)$ such that $\sigma(t)=-\omega^2 t$. Let $V=E_{f(t^6)}(L)\otimes_{\Z}\Q$. By Lemma ~\ref{CMT-2.1}, $\sigma$ acts $\Q(\omega)$-linearly on the $\Q(\omega)$-vector space $V$, and $V=V_1 \oplus V_{\zeta_6} \oplus \dots \oplus V_{\zeta_6^5}$.

    (a) We embed 
    $$E_{f(t^3)}(\Q(\omega)(t)) \hookrightarrow E_{f(t^6)}(\Q(\omega)(t)) \quad \text{ by } \quad (x(t),y(t)) \mapsto (x(t^2), y(t^2)).$$
    This map is well-defined and injective, and preserves the complex multiplication and the group structure (by observing the addition formula on elliptic curves).
    The image of this map is exactly $E_{f(t^6)}(\Q(\omega)(t^2))$.

    Let $v \in V$. Since $\langle \sigma^3 \rangle = \Gal(L/\Q(\omega)(t^2))$, 
    we have 
    $$v \in E_{f(t^6)}(\Q(\omega)(t^2))\otimes_{\Z}\Q \iff \sigma^3(v)=v.$$
    Let $v=\sum_{i=0}^5 c_i v_i$, where $c_i \in \Q$, $v_i \in V_{\zeta_6^i}$. 
    Then $\sigma^3(v)=\sum_{i=0}^5 (-1)^i c_i v_i$. 
    Therefore, $$\sigma^3(v)=v \iff v \in V_1 \oplus V_{\zeta_6^2} \oplus V_{\zeta_6^4}.$$
    By Lemma ~\ref{CMT-2.1}, the vector spaces $V_1$, $V_{\zeta_6^2},$ and $V_{\zeta_6^4}$ are identified with $E_{f(t)}(L)\otimes_{\Z}\Q, E_{t^2f(t)}(L)\otimes_{\Z}\Q, E_{t^4f(t)}(L)\otimes_{\Z}\Q$ respectively, which completes the proof of part (a).  

    (b) This time we embed 
    $$E_{f(t^2)}(\Q(\omega)(t)) \hookrightarrow E_{f(t^6)}(\Q(\omega)(t)) \quad \text{ by } \quad (x(t),y(t)) \mapsto (x(t^3), y(t^3)).$$
    Again, this map is well-defined and injective, and preserves the complex multiplication and the group structure. The image of this map is exactly $E_{f(t^6)}(\Q(\omega)(t^3))$.

    Let $v \in V$. 
    Since $\langle \sigma^2 \rangle = \Gal(L/\Q(\omega)(t^3))$, 
    we have 
    $$v \in E_{f(t^6)}(\Q(\omega)(t^3))\otimes_{\Z}\Q \iff \sigma^2(v)=v.$$
    Let $v=\sum_{i=0}^5 c_i v_i$, where $c_i \in \Q$, $v_i \in V_{\zeta_6^i}$. 
    Then $\sigma^2(v)=\sum_{i=0}^5 \omega^i c_i v_i$. 
    Therefore, $$\sigma^2(v)=v \iff v \in V_1 \oplus V_{\zeta_6^3}.$$
    By Lemma ~\ref{CMT-2.1}, the vector spaces $V_1$ and $V_{\zeta_6^3}$ are identified with $E_{f(t)}(L)\otimes_{\Z}\Q, E_{t^3f(t)}(L)\otimes_{\Z}\Q$, which completes the proof of part (b). 
\end{proof}

Now we prove some basic properties that will also be used later in our paper. 

\begin{lem} \label{basics}
    Let $f(t) \in \Q(t)$. We have:
    
    (a) $R(f(t))=R(t^{6n}f(t))$ for any $n \in \Z$;
    
    (b) $R(f(t))=R(f(1/t))$;
    
    (c) $R(f(t))=R(t^6f(1/t))$. 
\end{lem}
\begin{proof}
    (a) We define a map $E_{f(t)}(\Q(\omega)(t)) \to E_{t^{6n}f(t)}(\Q(\omega)(t))$ by $(x(t),y(t)) \mapsto (t^{2n}x(t),\allowbreak t^{3n}y(t))$. 
    This is a bijection and is also an isogeny between the two elliptic curves. Therefore $R(f(t))=R(t^{6n}f(t))$.

    (b) We define a map $E_{f(t)}(\Q(\omega)(t)) \to E_{f(1/t)}(\Q(\omega)(t))$ by $(x(t),y(t)) \mapsto (x(1/t), y(1/t))$. 
    This map is a bijection and preserves the complex multiplication and the group structure (by observing the addition formula). The same is true for its inverse. 
    Therefore $R(f(t))=R(f(1/t))$.

    (c) By part (a) and part (b), we have $R(f(t)) = R(f(1/t))=R(t^6f(1/t))$.
\end{proof}

\section{Low degree results} \label{section-Bremner}

Bremner \cite{Br} proved a formula for $r(f(t))$ for all $f(t) \in \Q[t]$ with $1 \leq \deg(f) \leq 3$.  
Their results are formulated in terms of the ``reduced form'' of $f(t)$ and an ``equivalence between polynomials'', whose definitions will be cited below. In this section, we cite their results and rewrite them explicitly in terms of the coefficients of $f(t)$. Note that any definition in this section will not be used in other sections.

\begin{prop} (\cite{Br} Theorem 1.7) \label{Br-1.7}
    Let $f(t) \in \Q[t]$ have degree $1$. Then $r(f(t))=0$.
\end{prop}

Now we move on to degree $2$.

\begin{defn} (\cite{Br} Section 1)
     Suppose that $f(t)=at^2+bt+c$ where $a \neq 0$ and $a,b,c \in \Q$. By translating $t$ by a rational number, we can achieve $b=0$. By scaling $t$ by a rational number, we can make $a$ square-free. We define the reduced form of $f(t)$ to be the polynomial after this change of variables.
    
    We define the equivalence between reduced forms by $at^2+c \sim at^2+c\alpha^6$ for $\alpha \in \Q\backslash\{0\}$. Note that $y^2=x^3+(at^2+c)$ and $y^2=x^3+(at^2+c\alpha^6)$ for some $\alpha\in\Q\backslash\{0\}$ define the same curve over $\Q$. We define two polynomials to be equivalent if their reduced forms are equivalent.
\end{defn}

\begin{prop} (\cite{Br} Theorem 1.5) \label{Br-1.5}
    Let $f(t) \in \Q[t]$ have degree $2$. Then $r(f(t))=1$ if and only if there exists $\lambda \in \Q\backslash\{0\}$ such that either $f(t) \sim t^2+\lambda^3$ or $f(t) \sim -3t^2+\lambda^3$. Otherwise, $r(f(t))=0$.
\end{prop}

Now we move on to degree $3$. Bremner split the case of degree $3$ into two subcases depending on whether $f(t)$ has a repeated root, and we will only cite and use their results when a repeated root exists.

\begin{defn} (\cite{Br} Section 1)
    Let $f(t) \in \Q[t]$ have degree $3$ and possess a squared linear factor in $\Q[t]$. By translation and scaling, $f(t)$ may be taken in the form $f(t)=t^2(at+b)$, $a,b \in \Q$, $a<0$ and $a$ is cube-free. 

    We define the equivalence between reduced forms by $t^2(at+b) \sim t^2(at+b\alpha^2)$ for $\alpha \in \Q\backslash\{0\}$. Note that $y^2=x^3+t^2(at+b)$ and $y^2=x^3+t^2(at+b\alpha^2)$ for some $\alpha\in\Q\backslash\{0\}$ define the same curve over $\Q$. We define two polynomials to be equivalent if their reduced forms are equivalent.
\end{defn}

\begin{prop} (\cite{Br} Theorem 1.3) \label{Br-1.3}
    Let $f(t) \in \Q[t]$ have degree $3$ and possess a squared linear factor in $\Q[t]$. 
    Then $r(f(t))=1$ if and only if there exists $\lambda \in \Q\backslash\{0\}$ such that either $f(t) \sim -t^3+\lambda^2 t^2$ or $f(t) \sim -t^3-3\lambda^2 t^2$. Otherwise, $r(f(t))=0$.
\end{prop}

Now we rewrite Bremner's results explicitly in terms of the coefficients of $f(t)$.

\begin{prop} \label{result-(2,1,0)} 
    Let $f(t)=At^2+Bt+C$ where $A,B,C \in \Q$ and $(A,B)\neq(0,0)$. Then $r(f(t))=1$ if and only if all of the following hold: \begin{enumerate}
        \item[(i)] $A \neq 0$;
    
        \item[(ii)] $B^2-4AC \neq 0$;
    
        \item[(iii)] $\sqrt{A}\in \Q(\omega)$; and
    
        \item[(iv)] $\frac{B^2-4AC}{4A}$ is the cube of a rational number, which is equivalent to being a cube in $\Q(\omega)$.
    \end{enumerate}

    If any of (i)--(iv) does not hold, then $r(f(t))=0$.
\end{prop}
\begin{proof}
    Proposition~\ref{Br-1.7} says that if $A=0$ (then $B \neq 0$) then $r(f(t))=0$. So it suffices to assume that $A \neq 0$ and prove that $r(f(t))=1$ if and only if all of (ii), (iii) and (iv) hold.

    We know that \begin{align*}
        f\left( t-\frac{B}{2A} \right) = At^2 + \frac{4AC-B^2}{4A}.
    \end{align*}
    So $f(t)$ is equivalent to 
    $t^2+\lambda^3$ or $-3t^2+\lambda^3$ for some $\lambda \in \Q\backslash\{0\}$ if and only if all of the following hold:
    
    (i) $A$ or $-3A$ is the square of a rational number;
    
    (ii) $4AC-B^2 \neq 0$; and 
    
    (iii) $\frac{B^2-4AC}{4A}$ is the cube of a rational number (because multiplying $\alpha^6$ for $\alpha \in \Q\backslash\{0\}$ does not change whether this number is a rational cube or not).

    We note that $A$ or $-3A$ being the square of a rational number implies that $\sqrt{A} \in \Q(\omega)$, and conversely, if $\sqrt{A}=a+b\sqrt{-3}$ for $a,b \in \Q$, then $A=a^2-3b^2+2ab\sqrt{-3}$, which implies that $ab=0$. If $a=0$, then $-3A$ is a square of a rational number. If $b=0$, then $A$ is a square of a rational number.
    So $A$ or $-3A$ being the square of a rational number is equivalent to $\sqrt{A} \in \Q(\omega)$.

    To see that a rational number being a cube in $\Q(\omega)$ implies that it is a cube in $\Q$, let $(a+b\sqrt{-3})^3$ for $a,b \in \Q$ be this rational number. Then $(a+b\sqrt{-3})^3 = (a^3-9ab^2) + \sqrt{-3}(3a^2b-3b^3)$, which tells us that $3a^2b-3b^3 = 3b(a-b)(a+b)=0$. If $b=0$, then $a^3-9ab^2=a^3$ is the cube of a rational number. If $b=\pm a$, then $a^3-9ab^2=-8a^3=(-2a)^3$ is also the cube of a rational number. So a rational number is the cube of an element in $\Q(\omega)$ if and only if it is the cube of an element in $\Q$. 
    
    Therefore, by Proposition ~\ref{Br-1.5}, we complete the proof of Proposition ~\ref{result-(2,1,0)}.
\end{proof}

\begin{prop} \label{result-(3,2)} 
    Let $f(t)=At^3+Bt^2$ where $A,B \in \Q$ and $(A,B)\neq(0,0)$. Then $r(f(t))=1$ if and only if all of the following hold: \begin{enumerate}
        \item[(i)] $A,B \neq 0$;
        \item[(ii)] $A$ is a cube of a rational number (or equivalently, a cube in $\Q(\omega)$); and
        \item[(iii)] $\sqrt{B} \in \Q(\omega)$. 
    \end{enumerate}

    If any of (i)--(iii) does not hold, then $r(f(t))=0$.
\end{prop}
\begin{proof}
    If $A=0$, then $f(t)=Bt^2$, and Proposition ~\ref{result-(2,1,0)} tells us that $r(f(t))=0$. 

    If $B=0$, then $f(t)=At^3$, and Proposition ~\ref{Br-1.3} tells us that $r(f(t))=0$.

    If $A,B \neq 0$, then Proposition ~\ref{Br-1.3} tells us that $r(f(t))=1$ if and only if both of the following hold: (i) $A$ is the cube of a rational number and (ii) $B$ or $-3B$ is the square of a (nonzero) rational number (because multiplying nonzero $\alpha^2$ to the degree-$2$ term does not change this property).

    In the proof of Proposition ~\ref{result-(2,1,0)}, we have proved that a rational number is a cube in $\Q$ if and only if it is a cube in $\Q(\omega)$. We have also proved that for a rational number $B$, we have $\sqrt{B} \in \Q(\omega)$ if and only if $B$ or $-3B$ is the square of a rational number. This completes the proof of Proposition ~\ref{result-(3,2)}.
\end{proof}

\begin{remark}
    The results in this section can also be proved alternatively by the method in our paper or in \cite{Kl}. See Remark ~\ref{alt-pf-(2,1,0)} and ~\ref{alt-pr-(3,2)}.
\end{remark}

\section{Proof of theorem \ref{result-(6,3,0)}} \label{section-(6,3,0)}

\begin{proof}

(a) By Lemma ~\ref{reduction}(a), we have 
\begin{align*}
R(At^6+Bt^3+C) & = R(At^2+Bt+C) + R(At^4+Bt^3+Ct^2) + R(At^6+Bt^5+Ct^4).
\end{align*}
Then by Lemma ~\ref{omega}, we complete the proof of part (a).

(b) By Proposition ~\ref{result-(2,1,0)}, we complete the proof of part (b).

(d) By Lemma ~\ref{omega} and Lemma ~\ref{basics}(c), we have 
\begin{align*}
    \rk_{\Z}E_3(\Q(t)) = \frac{1}{2} R(At^6+Bt^5+Ct^4)  = \frac{1}{2} R(Ct^2+Bt+A) = r(Ct^2+Bt+A).
\end{align*}
Then by Proposition ~\ref{result-(2,1,0)}, we complete the proof of part (d).

(c) Now we are going to prove the main part of Theorem ~\ref{result-(6,3,0)}, namely, part (c). In the first subsection below, we will prove the degenerate cases, namely, the first three cases in the statement of part (c). In the second subsection below, we will prove the non-degenerate case, namely, the last case in the statement of part (c).

\subsection{Degenerate Cases} \label{subsection-degen}

Now we prove the first three cases in part (c) of Theorem ~\ref{result-(6,3,0)}.

\begin{proof}

(i) If $A=0$, then $f(t)=Bt^3+Ct^2$. By Proposition ~\ref{result-(3,2)}, we complete the proof of the first case in part (c) of Theorem ~\ref{result-(6,3,0)}.

(ii) If $C=0$, then $f(t)=At^4+Bt^3$, and by Lemma ~\ref{basics}(c), we have $R(f(t))=R(Bt^3+At^2)$. By Proposition ~\ref{result-(3,2)}, we complete the proof of the second case in part (c) of Theorem ~\ref{result-(6,3,0)}.

(iii) If $A,C \neq 0$, $B^2-4AC=0$, then $f(t) = A t^2 (t+\frac{B}{2A})^2$. Then $f(t-\frac{B}{4A})= A (t - \frac{B}{4A})^2 (t + \frac{B}{4A})^2
=A (t^2 - \frac{B^2}{16A^2})^2=A (t^4 - \frac{B^2}{8A^2} t^2 + \frac{B^4}{256 A^4})
=A t^4 - \frac{B^2}{8A} t^2 + \frac{B^4}{256 A^3}$. 
By Lemma ~\ref{reduction}(b) and Lemma ~\ref{basics}(c), we have 
\begin{align*}
    R\left(A t^4 - \frac{B^2}{8A} t^2 + \frac{B^4}{256 A^3}\right) & = R\left(A t^2 - \frac{B^2}{8A} t + \frac{B^4}{256 A^3}\right) + R\left(A t^5 - \frac{B^2}{8A} t^4 + \frac{B^4}{256 A^3} t^3\right)\\
    & = R\left(A t^2 - \frac{B^2}{8A} t + \frac{B^4}{256 A^3}\right) + R\left(\frac{B^4}{256 A^3} t^3 - \frac{B^2}{8A} t^2 + At\right).
\end{align*} 

Since $A t^2 - \frac{B^2}{8A} t + \frac{B^4}{256 A^3} = A\left( t-\frac{B^2}{16A^2} \right)^2$, by Proposition ~\ref{result-(2,1,0)}, we have $R\left(A t^2 - \frac{B^2}{8A} t + \frac{B^4}{256 A^3}\right)\allowbreak=0$. 

Let \begin{align*}
    g(t)  = \frac{B^4}{256 A^3} t^3 - \frac{B^2}{8A} t^2 + At
     = At\left( \frac{B^2}{16A^2}t-1  \right)^2
\end{align*}
Then \begin{align*}
    g\left(t+\frac{16A^2}{B^2}\right) 
    = A \left(t+\frac{16A^2}{B^2}\right) \frac{B^4}{256A^4} t^2
    =\frac{B^4}{256A^3} t^3 + \frac{B^2}{16A} t^2.
\end{align*}
So $$R(f(t)) = R\left(  \frac{B^4}{256A^3} t^3 + \frac{B^2}{16A} t^2 \right).$$ 
Note that $\frac{B^3}{512A^3}$ is a rational cube and that $\frac{B^2}{16}$ is a rational square. Therefore, by Proposition ~\ref{result-(3,2)}, we have $R(f(t))=2r(f(t))=2$ if and only if $\sqrt[3]{2B} \in \Q(\omega)\backslash\{0\}$ and $\sqrt{A} \in \Q(\omega)\backslash\{0\}$. We also have $R(f(t))=0$ otherwise.

This completes the proof of the first three cases in part (c) of Theorem ~\ref{result-(6,3,0)}.
\end{proof}

\subsection{Non-degenerate case} \label{subsection-non-degen}

Now we prove the last case in part (c) of Theorem ~\ref{result-(6,3,0)}. Therefore, we assume that $A,C \neq 0$, $B^2-4AC \neq 0$ throughout this subsection.

In Lemma~\ref{lem-lin-alg}, we prove a simple linear algebra fact that will be used in the next two lemmas. In Lemma~\ref{rank-criteria-basis}, we show that when a group $G$ acts on a free $\Z[\omega]$-module $M$ of rank $n$, the rank of $M^G$ can be determined by the linear relations between $m_i-m_i^{\sigma}$ for a basis $\{m_i\}$ of $M$ and various $\sigma \in G$. 
In Lemma~\ref{rank-criteria}, we prove that the same is true when $\{m_i\}$ is only a basis up to finite index. Note that when $\rk_{\Z[\omega]}(M)=1$, Lemma~\ref{rank-criteria} reduces to \cite{Kl} Lemma 2.8.

In Lemma~\ref{lem-geo-rk-2} and Lemma~\ref{G-acts-on-M}, we prove that $M:=E_2(\bQ(t))$ is a free $\Z[\omega]$-module of rank $2$ and $G:=\Gal(\bQ/\Q(\omega))$ acts on the $\Z[\omega]$-module $M$.

In Proposition~\ref{gen-over-C} and Theorem~\ref{result-(4,3,2)}, we construct independent elements $P,Q \in E_2(\bQ(t))$, investigate the linear relations between $P-P^{\sigma}$ and $Q-Q^{\sigma}$ for a few $\sigma \in G$, and prove Theorem~\ref{result-(6,3,0)}. 

\begin{lem} \label{lem-lin-alg}
    Let $\{m_1, \dots, m_n\}$ be a basis of a free $\Z[\omega]$-module $M$ of rank $n$ and $1 \leq k \leq n$.
    Let $c_{ij} \in \Z[\omega]$ and $n_i=\sum_{j=1}^n c_{ij} m_j$. 
    Then $\{n_1, \dots, n_k\}$ is independent over $\Z[\omega]$ if and only if the matrix $(c_{ij})_{k \times n}$ over $\Q(\omega)$ has rank $k$. 
\end{lem}
\begin{proof}
    Let $t_1, \dots, t_k \in \Z[\omega]$. 
    Then $$\sum_{i=1}^k t_i n_i = \sum_{i=1}^k t_i \sum_{j=1}^n c_{ij}m_j = \sum_{j=1}^n m_j \sum_{i=1}^k t_i c_{ij}.$$
    So $$\sum_{i=1}^k t_i n_i=0 \iff \sum_{i=1}^k t_i c_{ij}=0 \ \forall 1\leq j \leq n \iff \sum_{i=1}^k t_i(c_{i1}, \dots, c_{in})=(0,\dots,0).$$
    Therefore, when $(c_{ij})_{k \times n}$ has rank $k$, i.e. has independent rows, we know that the above equivalence being true implies that all $t_i$'s are zero, which proves that $\{n_1, \dots, n_k\}$ is independent over $\Z[\omega]$. Conversely, if $\{n_1, \dots, n_k\}$ is independent over $\Z[\omega]$, then for any $t_1', \dots, t_k' \in \Q(\omega)$ such that $\sum_{i=1}^k t_i'(c_{i1}, \dots, c_{in})=(0,\dots,0)$, pick $t \in \Z[\omega]\backslash\{0\}$ such that $tt_i' \in \Z[\omega]$ for all $i$, then the above equivalence tells us that  $t_1'=\dots=t_k'=0$ and the matrix $(c_{ij})_{k \times n}$ has rank $k$.
\end{proof}

\begin{lem} \label{rank-criteria-basis}
    Let $M$ be a free $\Z[\omega]$-module of rank $n$ and $G$ be a group that acts on the $\Z[\omega]$-module $M$ (i.e. every $\sigma \in G$ acts $\Z[\omega]$-linearly).
    Let $\{m_1, \dots, m_n\}$ be a $\Z[\omega]$-basis for $M$. Then: \begin{enumerate}
        \item For $1 \leq k \leq n$, we have $\rk_{\Z[\omega]}(M^G) \geq k$ if and only if there exist $c_{ij} \in \Z[\omega]$ for $1 \leq i \leq k, 1 \leq j \leq n$, such that: \begin{enumerate}
            \item[(i)] the $k \times n$ matrix $(c_{ij})$ has rank $k$ over $\Q(\omega)$; and
            \item[(ii)] for each $1 \leq i \leq k$, we have $\sum_{j=1}^n c_{ij}(m_j-m_j^{\sigma})=0$ for all $\sigma \in G$.
        \end{enumerate}
        
        \item In particular, we have $\rk_{\Z[\omega]}(M^G)=n$ if and only if all of $m_1, \dots, m_n \in M^G$.
    \end{enumerate}
    
\end{lem}

\begin{proof} 

    (a) ($\Rightarrow$) When $\rk_{\Z[\omega]}(M^G) \geq k$, let $n_i = \sum_{j=1}^n c_{ij} m_j \in M^G$ ($1\leq i \leq k$) be $k$ independent elements in $M^G$. By Lemma~\ref{lem-lin-alg}, the matrix $(c_{ij})$ satisfies the conditions for part (a). 

    ($\Leftarrow$) When such $(c_{ij})$ exists, by Lemma~\ref{lem-lin-alg}, we know that $n_i:=\sum_{j=1}^n c_{ij} m_j$ are in $M^G$ and are independent over $\Z[\omega]$.

    (b) ($\Leftarrow$). This is obvious.

    ($\Rightarrow$). When $\rk_{\Z[\omega]}(M^G)=n$, let $C:=(c_{ij})_{n \times n}$ be the matrix of rank $n$ over $\Q(\omega)$ with $c_{ij} \in \Z[\omega]$ given in part (a). Still let $n_i=\sum_{j=1}^n c_{ij}m_j$, which lies in $M^G$. 
    Denote the column vectors $(n_1, \dots, n_n),(m_1, \dots, m_n) \in M^{\oplus n}$ by $\vec{n}, \vec{m}$ respectively. 
    Then we have $\vec{n}=C \vec{m}$. 
    Therefore 
    $$\mathrm{adj}(C) \vec{n} = \det(C) \vec{m}.$$ 
    Since $n_i \in M^G$ for all $i$, we have $\det(C)m_i \in M^G$ and therefore $\det(C)(m_i-m_i^{\sigma})=0$ for all $i$ and $\sigma \in G$.
    Since $C$ has rank $n$ and $M$ is torsion-free, we have $m_i-m_i^{\sigma}=0$ for all $\sigma \in G$. Therefore we have $m_i \in M^G$ for all $1 \leq i \leq n$.
\end{proof}

\begin{lem} \label{rank-criteria}
    Let $M$ be a free $\Z[\omega]$-module of rank $n$ and $G$ be a group that acts on the $\Z[\omega]$-module $M$. Let $\{m_1, \dots, m_n\}$ be a $\Z[\omega]$-independent subset of $M$. Then: \begin{enumerate}
        \item 
        For $1 \leq k \leq n$, we have $\rk_{\Z[\omega]}(M^G) \geq k$ if and only if there exist $c_{ij} \in \Z[\omega]$ for $1 \leq i \leq k, 1 \leq j \leq n$, such that: \begin{enumerate}
            \item[(i)] the $k \times n$ matrix $(c_{ij})$ has rank $k$ over $\Q(\omega)$; and
            \item[(ii)] for each $1 \leq i \leq k$, we have $\sum_{j=1}^n c_{ij}(m_j-m_j^{\sigma})=0$ for all $\sigma \in G$.
        \end{enumerate}
        \item 
        In particular, we have $\rk_{\Z[\omega]}(M^G)=n$ if and only if all of $m_1, \dots, m_n \in M^G$.
    \end{enumerate}
\end{lem}

\begin{proof}

    Let $\{m_1', \dots, m_n'\}$ be a basis for the $\Z[\omega]$-module $M$, then there exist $a_{ij} \in \Z[\omega]$ for $1 \leq i,j \leq n$, such that $m_i = \sum_{j=1}^n a_{ij} m_j'$. 
    Denote the column vectors $(m_1, \dots, m_n),(m_1', \dots, m_n') \in M^{\oplus n}$ by $\vec{m},\vec{m}'$ respectively, and the matrix $(a_{ij})_{n \times n}$ over $\Q(\omega)$ by $A$. 
    Then we have $$\vec{m}=A\vec{m}', \quad \mathrm{adj}(A) \vec{m}=\det(A)\vec{m}'.$$
    By Lemma~\ref{lem-lin-alg} and independence of $\{m_i\}$, we know that $A$ has rank $n$. 

    (a) By Lemma ~\ref{rank-criteria-basis}, it suffices to prove that the existence of the matrix $(c_{ij})_{k \times n}$ satisfying the conditions in part (a) over $\{m_1, \dots, m_n\}$ is equivalent to the existence of such a matrix satisfying the same conditions over $\{m_1', \dots, m_n'\}$. 

    Suppose that $C:=(c_{ij})_{k \times n}$ satisfies the conditions for $\{m_1, \dots, m_n\}$. Let $n_i = \sum_{j=1}^n c_{ij} m_j$ which lies in $M^G$. Denote the column vector $(n_1, \dots, n_k)$ by $\vec{n}$. Then 
    $$\vec{n}= C\vec{m}=CA \vec{m}'.$$ 
    Since $A$ is invertible and $n_i \in M^G$ for each $1 \leq i \leq k$,
    we know that the matrix $CA$ satisfies the conditions over $\{m_1', \dots, m_n'\}$.

    Conversely, suppose that $C:=(c_{ij})_{k \times n}$ satisfies the conditions for $\{m_1', \dots, m_n'\}$. Let $n_i=\sum_{j=1}^n c_{ij} m_j'$ which lies in $M^G$, and denote $(n_1, \dots, n_k)$ by $\vec{n}$. 
    Then 
    $$\det(A) \vec{n}=\det(A) C\vec{m}' = C \det(A) \vec{m}'=C \cdot \mathrm{adj}(A) \vec{m}.$$ 
    Since $\mathrm{adj}(A)$ is invertible over $\Q(\omega)$ and each entry of $\det(A)\vec{n}$ lies in $M^G$, we know that the matrix $C \cdot \mathrm{adj}(A)$ satisfies all the conditions over $\{m_1, \dots, m_n\}$. 
    
    (b) By Lemma ~\ref{rank-criteria-basis}, it suffices to prove that $m_i\in M^G$ for all $i$ if and only if $m_i' \in M^G$ for all $i$. 

    Since $\vec{m}=A\vec{m}'$, we know that $m_i'\in M^G$ for all $i$ implies that $m_i\in M^G$ for all $i$.

    Conversely, since $\mathrm{adj}(A)\vec{m}= \det(A)\vec{m}'$, we know that $m_i\in M^G$ for all $i$ implies that $\det(A)m_i' \in M^G$ for all $i$,
    which further implies that $\det(A)(m_i'-m_i'^{\sigma}) =0$ for all $i$ and all $\sigma \in G$. Since $M$ is torsion-free and $A$ is invertible, we know that $m_i'\in M^G$ for all $i$.
\end{proof}

\begin{remark} \label{indep-of-sigma}
    Note that the matrix $(c_{ij})$ in part (a) does not depend on the choice of $\sigma \in G$.
\end{remark}

Now we prove in Lemma~\ref{lem-geo-rk-2} that the geometric Mordell-Weil group $E_2(\bQ(t))$ is a free $\Z[\omega]$-module of rank $2$ by using the Shioda-Tate formula as stated in \cite{SS} Section 7.2 and the method in \cite{Kl} Lemma 2.2, and check in Lemma~\ref{G-acts-on-M} that the group $G:=\Gal(\bQ/\Q(\omega))$ acts on the $\Z[\omega]$-module $M$.

\begin{lem} \label{lem-geo-rk-2}
    Recall that $E_2: y^2=x^3+At^4+Bt^3+Ct^2$ and that $A,C,B^2-4AC \neq 0$. Then the $\Z[\omega]$-module $E_2(\bar{\Q}(t))$ is free of rank $2$.
\end{lem}
\begin{proof}
    Since $A,C, B^2-4AC \neq 0$, we have $f(t) = A t^2 (t-\alpha)(t-\beta)$ for some distinct nonzero $\alpha,\beta \in \bQ$. Therefore its singular fibres are: at $t=0$ of type IV, at $t=\alpha$ of type II, at $t=\beta$ of type II, and at $t=\infty$ of type IV. By \cite{SS} Section 7.2, we know that the $\Z$-rank of $E_2(\bar{\Q}(t))$ is $4$. Once we show that $E_2(\bQ(t))$ is a free $\Z[\omega]$-module, we know that its rank as a $\Z[\omega]$-module is $2$.

    First we show that $E_2(\bar{\Q}(t))$ is a free $\Z$-module. We embed it into $E_0(\bQ(t))$ where $E_0:y^2=x^3+At^6+Bt^3+C$ by the decomposition of $E_0(\bQ(t))$ in Section ~\ref{section-decomposition}.
    Since $A,C,B^2-4AC \neq 0$, we know that the singular fibres of $E_0$ are at 
    $t=\sqrt[3]{\alpha}, \sqrt[3]{\alpha} \omega, \sqrt[3]{\alpha} \omega^2, \sqrt[3]{\beta}, \sqrt[3]{\beta} \omega, \sqrt[3]{\beta} \omega^2$, and all these (distinct) singular fibres are of type II. 
    Hence $E_0(\bQ(t))$ has rank $8$ by \cite{SS} Section 7.2. By Theorem 8.8 in \cite{SS}, the geometric Mordell-Weil group (i.e. $E_0(\bQ(t))$) of a rational elliptic surface of rank $8$ is free abelian. Hence its submodule $E_2(\bQ(t))$ is a free $\Z$-module.

    Now we show that $E_2(\bQ(t))$ is a free $\Z[\omega]$-module. Since $\Z[\omega]$ is a PID, it suffices to show that the module is torsion-free. We use the same proof as in Lemma 2.2 in \cite{Kl}. Let $P \in E_2(\bQ(t)) \backslash\{O\}$, $a,b \in \Z$, and suppose that $(a+b\omega)P=O$. Then $(a^2+b^2-ab)P=(a+b\omega^2)(a+b\omega)P=O$. Since $E_2(\bQ(t))$ is torsion-free as $\Z$-module, we have $(a-\frac{1}{2}b)^2+\frac{3}{4}b^2 = a^2+b^2-ab=0$. So $a=b=0$.
\end{proof}

\begin{lem} \label{G-acts-on-M}
    Let $M=E_2(\bQ(t))$ which is a free $\Z[\omega]$-module of rank $2$ by Lemma~\ref{lem-geo-rk-2}. 
    Let $G=\Gal(\bQ/\Q(\omega))$ and act on $E_2$ by the usual Galois action on elliptic curves. Then $G$ acts on the $\Z[\omega]$-module $M$, i.e., every $\sigma \in G$ acts $\Z[\omega]$-linearly on $M$. 
\end{lem}
\begin{proof}
    Let $\sigma \in G$. Then $\sigma(\omega)=\omega$.
    
    Since $G$ is a Galois action on elliptic curves, we have $\sigma(m+n)=\sigma(m)+\sigma(n)$ for all $m,n \in M$ and $\sigma(cm)=c\sigma(m)$ for all $m \in M, c \in \Z$.

    For $m=(x,y)$, we have $\sigma(\omega m)=\sigma(\omega(x,y))=\sigma((\omega x,y)) = (\sigma(\omega x),\sigma(y))=(\sigma(\omega)\sigma(x),\sigma(y))\allowbreak
    =(\omega \sigma(x),\sigma(y))=\omega(\sigma(x),\sigma(y))=\omega \sigma(m)$.
    
    Therefore, we have $\sigma((a+b\omega)m) = \sigma(am + b\omega m) = \sigma(am)+\sigma(b\omega m) = a\sigma(m) + b \sigma(\omega m) = a\sigma(m) + b \omega \sigma(m) = (a+b\omega)\sigma(m)$ for all $a,b \in \Z$. This completes the proof of Lemma ~\ref{G-acts-on-M}.
\end{proof}

Now we construct two elements in the rank-two $\Z[\omega]$-module $E_2(\bQ(t))$ and prove that they are $\Z[\omega]$-independent.

\begin{prop} \label{gen-over-C}
    Let $A,B,C \in \Q$,  $A,C,B^2-4AC \neq 0$, and recall that $E_2: y^2 = x^3 + At^4 + Bt^3 + Ct^2$. Fix any square root of $A$ and any square root of $C$ in $\bQ$ and denote $\sqrt{AC}=\sqrt{A}\sqrt{C}$. Then 
    the points 
    $$P=\left(\sqrt[3]{2 \sqrt{A}\sqrt{C}-B} \cdot t, \quad   \sqrt{A} t^2 + \sqrt{C} t\right)$$
    and 
    $$Q=\left(\sqrt[3]{-2 \sqrt{A}\sqrt{C}-B} \cdot t, \quad   \sqrt{A} t^2 - \sqrt{C} t\right)$$
    are on $E_2(\bQ(t))$ and are independent over $\Z[\omega]$, where the two choices of cube roots are two arbitrarily fixed choices.
\end{prop}

\begin{proof}
    It is easy to check that both $P$ and $Q$ do lie on $E_2$. Now we prove the $\Z[\omega]$-independence.

    Since $At^4+Bt^3+Ct^2=t^2(At^2+Bt+C)$, $B^2-4AC \neq 0$, $A \neq 0$, $C \neq 0$, and $6-\deg(At^4+Bt^3+Ct^2)=2$, we know that the singular fibres are at $t=0$, $t = \infty$, and $t$ being the two distinct roots of $At^2+Bt+C=0$. At the two simple roots of $At^2+Bt+C=0$, the fibres are of type II. At $t=0$ and $t=\infty$, the fibres are of type IV.

    Now we follow Tate's algorithm (Section 5.8.2 in \cite{SS}) to find the fibre at $t=0$ after the resolution of singularities. The surface $y^2=x^3+At^4+Bt^3+Ct^2$ is singular at $(x,y,t)=(0,0,0)$, and is smooth at $(x,y,0)$ when $(x,y) \neq (0,0)$. To find the $t$-chart of the first blow-up, we let $x=x_1t$ and $y=y_1t$, then the surface after the blow-up is $\tilde{E}: y_1^2=tx_1^3+At^2+Bt+C$. Let $F(x_1,y_1,t)=tx_1^3+At^2+Bt+C-y_1^2$, then $\frac{\partial F}{\partial x_1}=\frac{\partial F}{\partial y_1}=\frac{\partial F}{\partial t}=0$ implies that $y_1=0$, which is impossible when $t=0$ because $C \neq 0$. Therefore the surface $\tilde{E}$ is smooth in the $t$-chart after the blow-up. Letting $t=0$ in the equation defining $\tilde{E}$, we know that the fibre at $t=0$ is $y_1^2=C$. Therefore, we get two components at the fibre $t=0$ as  
    $$\{((0,0,0), (x_1:y_1:t_1)) \ \big| \ y_1=\sqrt{C} t_1 \} \quad \text{ and } \quad \{((0,0,0), (x_1:y_1:t_1))\ \big| \ y_1=-\sqrt{C} t_1 \}$$ 
    in the $((x,y,t),(x_1:y_1:t_1))$-coordinates. These two components do not intersect $(O)$, and a type IV singularity has only two non-identity components, so these are the two non-identity components and we do not have to look at other charts after the blow-up.

    Now we blow-up the sections $P$ and $Q$. In the $t$-chart of the blow-up, we have $x_1=\frac{x}{t}$ and $y_1=\frac{y}{t}$. Plugging in the $x$-coordinate and $y$-coordinate of $P$, and $t=0$, we know that the blow-up of the section $P$, denoted by $\tilde{P}$, intersects the fibre $t=0$ at  

    $$\tilde{P}\big|_{t=0}=\left( (0,0,0), \left( \sqrt[3]{2 \sqrt{A}\sqrt{C}-B} : \sqrt{C}  : 1  \right) \right).$$
    Similarly, 
    $$\tilde{Q}\big|_{t=0}=\left( (0,0,0), \left(\sqrt[3]{-2 \sqrt{A}\sqrt{C}-B} : -\sqrt{C}  : 1  \right) \right).$$
    Since $C \neq 0$, we know that $\tilde{P}$ and $\tilde{Q}$ intersect at different non-identity components at the fibre of $t=0$.

    To find the components at $t=\infty$, we need the admissible transformation $(x(t),y(t)) \mapsto (s^2 x(\frac{1}{s}), s^3 y(\frac{1}{s}))$ to switch to the part of the surface around $t=\infty$, in other words $s=0$ (see \cite{SS} Section 3.4.5). 
    Then the surface becomes 
    $$E_2:y^2=x^3+Cs^4+Bs^3+As^2,$$ and the sections $P,Q$ become 
    $$P=\left(\sqrt[3]{2 \sqrt{A}\sqrt{C}-B} \cdot s, \quad   \sqrt{C} s^2 + \sqrt{A} s\right)$$
    and 
    $$Q=\left(\sqrt[3]{-2 \sqrt{A}\sqrt{C}-B} \cdot s, \quad   -\sqrt{C} s^2 + \sqrt{A} s\right).$$

    With the same calculation as we did for $t=0$, we know that the two non-identity components at the fibre of $s=0$, in the $((x,y,s),(x_1:y_1:s_1))$-coordinates, are 
    $$\{((0,0,0), (x_1:y_1:s_1)) \ \big| \ y_1=\sqrt{A} s_1 \} \quad \text{ and } \quad \{((0,0,0), (x_1:y_1:s_1))\ \big| \ y_1=-\sqrt{A} s_1 \}.$$
    The blow-ups of the sections $P,Q$ at $s=0$ in the $s$-chart are  
    $$\tilde{P}\big|_{s=0}=\left( (0,0,0), \left( \sqrt[3]{2 \sqrt{A}\sqrt{C}-B} : \sqrt{A}  : 1  \right) \right)$$
    and 
    $$\tilde{Q}\big|_{s=0}=\left( (0,0,0), \left(\sqrt[3]{-2 \sqrt{A}\sqrt{C}-B} : \sqrt{A}  : 1  \right) \right).$$
    Since $A,C \neq 0$, we know that $\tilde{P}, \tilde{Q}$ intersect at different points with the fibre at $s=0$. 
    By our calculation of the non-identity components, $\tilde{P}$ and $\tilde{Q}$ intersect at the same non-identity component of the fibre at $s=0$.

    Therefore, by Table 6.1 in \cite{SS}, we have \begin{align*}
        \sum_v contr_v(P,Q)& =\frac{1}{3}+\frac{2}{3} = 1;\\
        \sum_v contr_v(P)& =\frac{2}{3}+\frac{2}{3} = \frac{4}{3};\\
        \sum_v contr_v(Q)& =\frac{2}{3}+\frac{2}{3} = \frac{4}{3}.
    \end{align*}

    Since the coordinates of $P$ and $Q$ in both $\{t \neq \infty\}$ and $\{s \neq \infty\}$ are polynomials in $t$ and $s$ respectively, we have $(P \cdot O)=(Q \cdot O)=0$.
    Since $A,C \neq 0$, we have $\sqrt[3]{2 \sqrt{A}\sqrt{C}-B} \cdot t \neq \sqrt[3]{-2 \sqrt{A}\sqrt{C}-B} \cdot t$ for $t \neq 0,\infty$. Therefore $P,Q$ do not intersect at $t \neq 0,\infty$. We have seen in our calculation of $\tilde{P}$ and $\tilde{Q}$ that they intersect at different points on both the fibre at $t=0$ and the fibre at $s=0$. Therefore $(P \cdot Q)=0$.

    Recall Shioda's height-pairing formulas (Theorem 6.24 in \cite{SS}) that 
    $$\langle P,Q\rangle = \chi + (P \cdot O) + (Q \cdot O) - (P \cdot Q) - \sum_v contr_v(P,Q)$$
    and that
    $$\langle P,P \rangle = 2 \chi + 2 (P \cdot O) - \sum_v contr_v(P).$$
    Note also that $\chi=1$ here because $E_2$ is a rational elliptic surface. Therefore 
    \begin{align*}
        \langle P,P \rangle & =2+0-\frac{4}{3}=\frac{2}{3},\\
        \langle P,Q \rangle & = 1+0+0-0-1=0,\\
        \langle Q,Q \rangle & =2+0-\frac{4}{3}=\frac{2}{3}.
    \end{align*}

    Now we find $\langle P, \omega P\rangle$. 
    In $\{t \neq \infty\}$, we have 
    $\omega P  = \left(\omega\sqrt[3]{2 \sqrt{A}\sqrt{C}-B} \cdot t, \   \sqrt{A} t^2 + \sqrt{C} t\right).$
    In $\{s \neq \infty\}$, we have 
    $\omega P  = \left(\omega\sqrt[3]{2 \sqrt{A}\sqrt{C}-B} \cdot s, \ \sqrt{C} s^2 + \sqrt{A} s\right)$. 
    Therefore, the blow-up of $\omega P$ at $t=0$ in the $((x,y,t),(x_1:y_1:t_1))$-coordinates is
    $$\tilde{\omega P}\big|_{t=0} = \left( (0,0,0), \left( \omega\sqrt[3]{2 \sqrt{A}\sqrt{C}-B} : \sqrt{C}  : 1  \right) \right),$$
    and at $s=0$ in the $((x,y,s),(x_1:y_1:s_1))$-coordinates is 
    $$\tilde{\omega P}\big|_{s=0} = \left( (0,0,0), \left( \omega\sqrt[3]{2 \sqrt{A}\sqrt{C}-B} : \sqrt{A}  : 1  \right) \right).$$

    Since $B^2 -4AC \neq 0$, we have $\sqrt[3]{2\sqrt{A}\sqrt{C}-B} \neq 0$. Therefore, $\tilde{\omega P}$ and $\tilde{P}$ have different $x_1$-coordinates at $t=0$ and $s=0$, and different $x$-coordinates at $t \neq 0,\infty$. So $(P \cdot \omega P) = 0$. We have calculated that $(P \cdot O)=0$, and the same calculation gives $(\omega P \cdot O)=0$. 
    From our calculation of $\tilde{P}$ and $\tilde{\omega P}$, we can see that they intersect at the same non-identity component at both $t=0$ and $s=0$. So $\sum_v contr_v(P,\omega P) = \frac{2}{3}+\frac{2}{3} = \frac{4}{3}$. 
    Therefore, 
    $$\langle P, \omega P \rangle = 1 + 0 + 0 - 0 - \frac{4}{3}=-\frac{1}{3}.$$

    The same computation gives $\langle Q, \omega Q\rangle = - \frac{1}{3}$, as taking the other square root of $C$ does not change the computation. 
    We also notice that 
    $\omega Q=\left(\omega\sqrt[3]{-2 \sqrt{A}\sqrt{C}-B} \cdot t, \   \sqrt{A} t^2 - \sqrt{C} t\right),$
    so the only difference between $Q$ and $\omega Q$ is that $x_Q$ and $x_{\omega Q}$ take different cube roots of $-2 \sqrt{A}\sqrt{C}-B$. The same thing happens to $P$. Note that 
    in the computation of $\langle P, Q\rangle$, whichever cube root we take to get $x_Q$ and $x_P$ gives the same result for $\langle P,Q \rangle$, because the expression inside the cube root for $x_P$ and for $x_Q$ are not equal. This gives $\langle P, \omega Q \rangle = 0$, $\langle \omega P, Q \rangle = 0$, and $\langle \omega P, \omega Q \rangle = 0$. 

    Therefore the height pairing matrix with respect to $P, \omega P, Q, \omega Q$ is 
    $$\begin{pmatrix}
        \frac{2}{3} & - \frac{1}{3} & 0 & 0\\
        - \frac{1}{3} & \frac{2}{3} & 0 & 0\\
        0 & 0 & \frac{2}{3} & - \frac{1}{3}\\
        0 & 0 & - \frac{1}{3} & \frac{2}{3}\\
    \end{pmatrix},$$
    which has full rank. 
    Hence $P,Q$ are independent over $\Z[\omega]$.
\end{proof}

Now we are ready to find the rank of $E_2(\Q(t))$.

\begin{thm} \label{result-(4,3,2)}
    Let $A,B,C \in \Q$,  $A,C,B^2-4AC \neq 0$, and recall that $E_2: y^2 = x^3 + At^4 + Bt^3 + Ct^2$. Then: \begin{enumerate}
    \item[(a)] $\rk_{\Z} E_2(\Q(t))=2$ if and only if both of the following hold: \begin{enumerate}
        \item[(i)] both of $\sqrt{A}, \sqrt{C}$ lie in $\Q(\omega)$;
        \item[(ii)] both of $\sqrt[3]{2\sqrt{A}\sqrt{C}-B}$ and $\sqrt[3]{-2\sqrt{A}\sqrt{C}-B}$ lie in $\Q(\omega)$.
    \end{enumerate} 
    \item[(b)] $\rk_{\Z} E_2(\Q(t))=1$ if and only if (at least) one of the following holds: \begin{enumerate}
        \item[(i)] Exactly one of $\sqrt{A},\sqrt{C}$ lies in $\Q(\omega)$, and both of $\sqrt[3]{2\sqrt{A}\sqrt{C}-B},\sqrt[3]{-2\sqrt{A}\sqrt{C}-B}$ lie in $\Q(\omega,\sqrt{AC})$;
        \item[(ii)] Both of $\sqrt{A},\sqrt{C}$ lie in $\Q(\omega)$, and exactly one of $\sqrt[3]{2\sqrt{A}\sqrt{C}-B},\sqrt[3]{-2\sqrt{A}\sqrt{C}-B}$ lies in $\Q(\omega)$.
    \end{enumerate}
    \item[(c)] $\rk_{\Z} E_2(\Q(t))=0$ otherwise.
    \end{enumerate}
\end{thm}

\begin{proof}
    Let $G=\Gal(\bQ/\Q(\omega))$, and $M=E_2(\bQ(t))$, and $P,Q \in E_2(\bQ(t))$ as constructed in Proposition ~\ref{gen-over-C}, where we fix an arbitrary choice of the two square roots and the two cube roots. Then by Lemma ~\ref{G-acts-on-M}, we know that these $M$ and $G$ satisfy the conditions in Lemma ~\ref{rank-criteria}. 
    We also know by Lemma ~\ref{omega} that $\rk_{\Z}(E_2(\Q(t)))
    =\frac{1}{2} \rk_{\Z}(E_2(\Q(\omega)(t)))
    =\frac{1}{2} \rk_{\Z} (M^G)
    =\rk_{\Z[\omega]}(M^G)$. 

    (a) This is precisely the equivalent conditions for $P,Q \in E_2(\Q(\omega)(t))$. By Lemma ~\ref{rank-criteria}, we get part (a).

    (b) Now we find the equivalent conditions in terms of $A,B,C$ for the following statement: \begin{equation} \label{rk-1-equiv}
        \exists c_1, c_2 \in \Z[\omega], (c_1, c_2) \neq (0,0) \ s.t.\ \forall \sigma \in G, \ \text{we have} \ c_1(P-P^{\sigma})+c_2(Q-Q^{\sigma})=0.
    \end{equation}
    This statement is equivalent to $\rk_{\Z[\omega]}(M^G) \geq 1$ by Lemma~\ref{rank-criteria}.

    Denote $\alpha=\sqrt[3]{2 \sqrt{A}\sqrt{C}-B}$ and $\beta=\sqrt[3]{-2 \sqrt{A}\sqrt{C}-B}$. 
    We first find all the possibilities for the extension degrees between $\Q(\omega,\alpha), \Q(\omega,\beta), \Q(\omega,\sqrt{AC}),\Q(\omega)$.

    First, we discuss the degree of $\alpha$ over $\Q(\omega,\sqrt{AC})$. If the polynomial $p_1(x)=x^3-(2\sqrt{AC}-B)$ is irreducible over $\Q(\omega,\sqrt{AC})$, then $\alpha$ has degree $3$ over $\Q(\omega,\sqrt{AC})$. 
    If the degree-$3$ polynomial $p_1$ is reducible over $\Q(\omega,\sqrt{AC})$, then it contains a root in $\Q(\omega,\sqrt{AC})$ and therefore contains all three roots by multiplying $\omega$ and $\omega^2$ to the known root. In this case, $\alpha$ has degree $1$ over $\Q(\omega,\sqrt{AC})$. Same thing happens for $\beta$. 
    Therefore in any case, we have 
    $$[\Q(\omega,\alpha):\Q(\omega,\sqrt{AC})] \in \{1,3\} \quad \text{and} \quad [\Q(\omega,\beta):\Q(\omega,\sqrt{AC})] \in \{1,3\}.$$
    
    Second, we prove that when $\sqrt{AC} \notin \Q(\omega)$, we have \begin{equation} \label{(2)}
    [\Q(\omega,\alpha):\Q(\omega,\sqrt{AC})]=[\Q(\omega,\beta):\Q(\omega,\sqrt{AC})].
    \end{equation}

    If $[\Q(\omega,\alpha):\Q(\omega,\sqrt{AC})]=1$, then $2\sqrt{AC}-B = (a+b\sqrt{AC})^3=(a^3+3ab^2AC) + (3a^2b+b^3AC)\sqrt{AC}$ for some $a,b \in \Q(\omega)$. 
    Since $\sqrt{AC} \notin \Q(\omega)$, we have $a^3+3ab^2AC=-B$ and $3a^2b+b^3AC=2$. Therefore, $(a-b\sqrt{AC})^3=(a^3+3ab^2AC) - (3a^2b+b^3AC)\sqrt{AC} = -2\sqrt{AC}-B = \beta^3$. Hence $\beta=\omega^i (a-b\sqrt{AC}) \in \Q(\omega,\sqrt{AC})$ for some $i \in \{0,1,2\}$. This means that $[\Q(\omega,\beta):\Q(\omega,\sqrt{AC})]=1=[\Q(\omega,\alpha):\Q(\omega,\sqrt{AC})]$.

    Similarly, if $[\Q(\omega,\beta):\Q(\omega,\sqrt{AC})]=1$, we have $[\Q(\omega,\alpha):\Q(\omega,\sqrt{AC})]=1$. If none of $[\Q(\omega,\alpha):\Q(\omega,\sqrt{AC})]$ and $[\Q(\omega,\beta):\Q(\omega,\sqrt{AC})]$ is $1$, then $[\Q(\omega,\alpha):\Q(\omega,\sqrt{AC})]=[\Q(\omega,\beta):\Q(\omega,\sqrt{AC})]=3$.
    Therefore, (\ref{(2)}) holds when $\sqrt{AC} \notin\Q(\omega)$.

    Hence, cases A--F shown in Figure~\ref{cubic-cases} are all possibilities for the extension degrees between $\Q(\omega,\alpha)$, $\Q(\omega,\beta)$, $\Q(\omega,\allowbreak\sqrt{AC})$ and $\Q(\omega)$. On the other hand, cases (i)--(v) shown in Figure~\ref{quadratic-cases} are all possibilities for the extension degrees between $\Q(\omega,\sqrt{A},\sqrt{C})$, $\Q(\omega,\sqrt{A})$, $\Q(\omega,\sqrt{C})$ and $\Q(\omega)$. 
    Now we discuss these cases.\\

\input{Galois-diagrams.tex}

    \begin{figure}
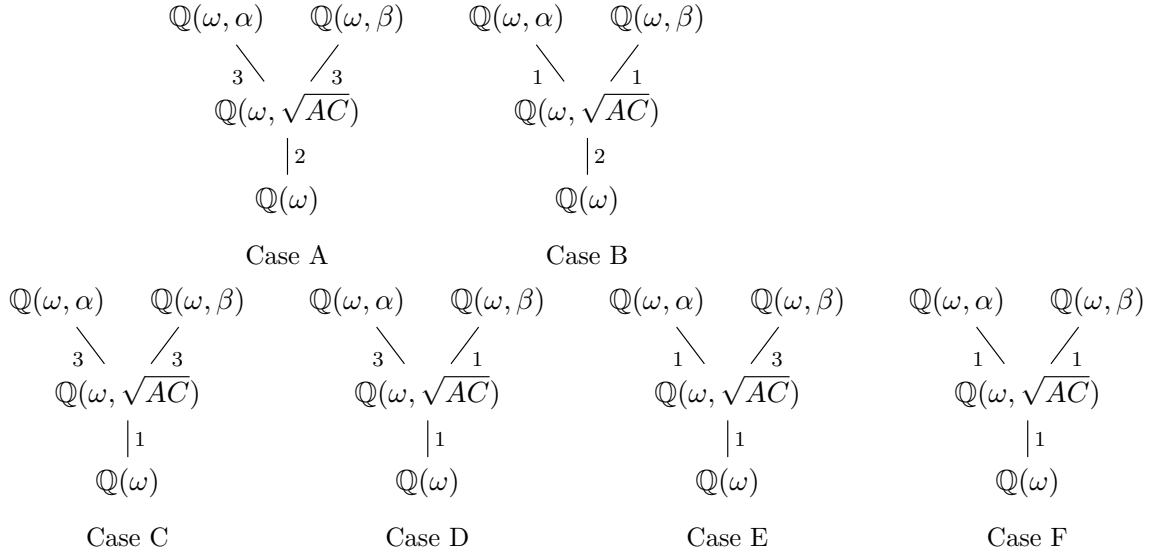

    \centering
    \hspace*{-0.24\textwidth}
    \CubicDiagram{A}{3}{3}{2}
    \CubicDiagram{B}{1}{1}{2}

    \CubicDiagram{C}{3}{3}{1}
    \CubicDiagram{D}{3}{1}{1}
    \CubicDiagram{E}{1}{3}{1}
    \CubicDiagram{F}{1}{1}{1}

    \caption{Cases regarding the cubic extensions.}
    \label{cubic-cases}
    \end{figure}

    \begin{figure}
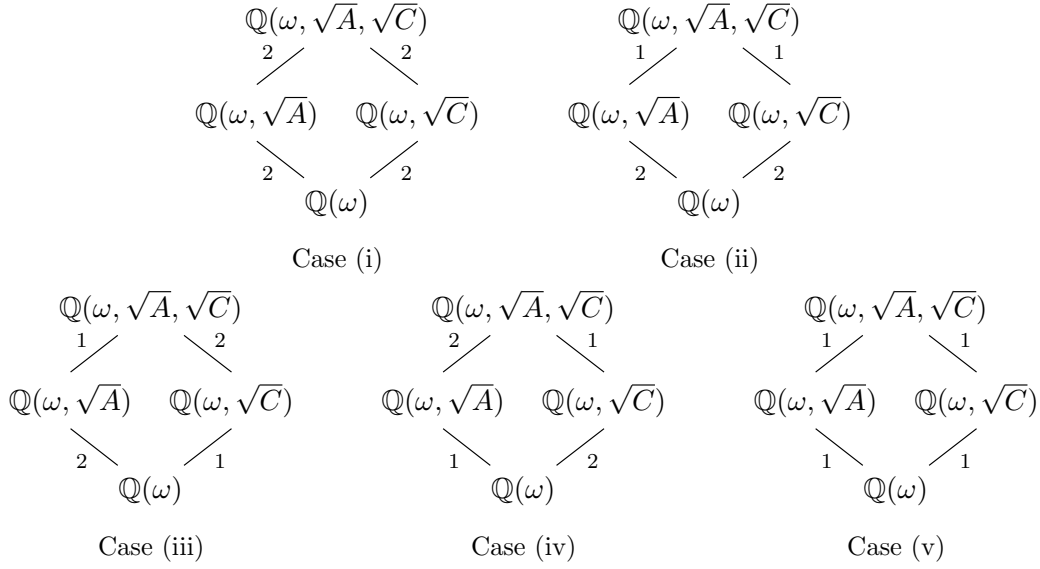

    \centering
    \QuadraticDiagram{i}{2}{2}{2}{2}
    \QuadraticDiagram{ii}{1}{1}{2}{2}

    \QuadraticDiagram{iii}{1}{2}{2}{1}
    \QuadraticDiagram{iv}{2}{1}{1}{2}
    \QuadraticDiagram{v}{1}{1}{1}{1}

    \caption{Cases regarding the quadratic extensions.}
    \label{quadratic-cases}
    \end{figure}

    We first show that, in case (i) or case (ii), there exists $\sigma \in G$ such that $\begin{cases}
        \sigma(\sqrt{A})=-\sqrt{A}\\
        \sigma(\sqrt{C})=-\sqrt{C}.
    \end{cases}$

    In case (i), $\Q(\omega,\sqrt{A},\sqrt{C})$ has degree $4$ over $\Q(\omega)$, and it is Galois over $\Q(\omega)$ since it is the splitting field of $(x^2-A)(x^2-C)$. Thus, there exists an element in $\Gal(\Q(\omega,\sqrt{A},\sqrt{C})/\Q(\omega))$ that sends both $\sqrt{A}$ and $\sqrt{C}$ to their negatives. Extending this element, we know that there exists $\sigma \in G$ such that $\sigma(\sqrt{A})=-\sqrt{A}$ and $\sigma(\sqrt{C})=-\sqrt{C}$.

    In case (ii), we extend the nontrivial element in $\Gal(\Q(\omega,\sqrt{A})/\Q(\omega))$ to an element $\sigma \in G$. Then $\sigma(\sqrt{A})=-\sqrt{A}$. Let $\sqrt{C}=a+b\sqrt{A}$ for $a,b \in \Q(\omega)$. Then $C=(a^2+b^2A)+2ab\sqrt{A}$. Since $\sqrt{A} \notin \Q(\omega)$, we have $ab=0$. Since $\sqrt{C} \notin \Q(\omega)$, we have $b \neq 0$. Therefore $a=0$ and $\sqrt{C}=b\sqrt{A}$. Hence $\sigma(\sqrt{C})=-\sqrt{C}$. Therefore we get an element $\sigma \in G$ such that $\sigma(\sqrt{A})=-\sqrt{A}$ and $\sigma(\sqrt{C})=-\sqrt{C}$.

    Now we show that the existence of $\sigma \in G$ such that $\sigma(\sqrt{A})=-\sqrt{A}$ and $\sigma(\sqrt{C})=-\sqrt{C}$ implies that Statement~(\ref{rk-1-equiv}) does not hold.
    Suppose that $(c_1,c_2)$ is a pair for (\ref{rk-1-equiv}). 
    Then we have $\sigma(\sqrt{AC})=(-\sqrt{A})(-\sqrt{C})=\sqrt{AC}$, so $\sigma$ fixes both $\alpha^3$ and $\beta^3$. Therefore $\sigma(x_P)=\omega^i x_P$ and $\sigma(x_Q)=\omega^j x_Q$ for some $i,j \in \{0,1,2\}$, where $x_P,x_Q$ denote the $x$-coordinates of $P,Q$. We also have $\sigma(y_P)=-y_P$ and $\sigma(y_Q)=-y_Q$. Therefore, $\sigma(P)=-\omega^i P$ and $\sigma(Q)=-\omega^j Q$. Hence $c_1(P+\omega^iP)+c_2(Q+\omega^jQ)=0$. Since $\mathrm{span}_{\Z[\omega]}\{P\} \cap \mathrm{span}_{\Z[\omega]}\{Q\} = \{0\}$, we have $c_1(1+\omega^i)=c_2(1+\omega^j)=0$ and therefore $c_1=c_2=0$. This proves that in either case (i) or case (ii), Statement~(\ref{rk-1-equiv}) does not hold.\\

    Before discussing the remaining cases (iii), (iv), (v), we first prove a claim. 

    \begin{claim-star} 
        Suppose that $(c_1, c_2)$ is a pair that satisfies Statement (\ref{rk-1-equiv}). If $[\Q(\omega,\alpha):\Q(\omega,\sqrt{AC})]=3$, then $c_1=0$. Similarly, if $[\Q(\omega,\beta):\Q(\omega,\sqrt{AC})]=3$, then $c_2=0$.
    \end{claim-star}
    \begin{proof}
        Suppose that $\alpha$ has degree $3$ over $\Q(\omega,\sqrt{AC})$, then there exists $\sigma\in \Gal(\bQ/\Q(\omega,\sqrt{AC}))$ such that $\sigma(\alpha) \neq \alpha$. Since $\sigma(\alpha)^3 =\sigma(2\sqrt{AC}-B)=2\sqrt{AC}-B=\alpha^3$, we have $\sigma(\alpha)=\omega^i \alpha$ for some $i \in \{1,2\}$. Since $\sigma(\beta)^3=\beta^3$, we have $\sigma(\beta)=\omega^j \beta$ for some $j \in \{0,1,2\}$. 
        Since $\sigma(\sqrt{AC})=\sqrt{AC}$, we have either $\sigma(\sqrt{A})=\sqrt{A}, \sigma(\sqrt{C})=\sqrt{C}$, or $\sigma(\sqrt{A})=-\sqrt{A}, \sigma(\sqrt{C})=-\sqrt{C}$. 
        In either case, we have $\sigma(P)=\zeta_6^k P$ and $\sigma(Q)=\zeta_6^l Q$ for some $k \in \{1,2,4,5\}, l \in \{0, \dots, 5\}$, where $\zeta_6=e^{\frac{2\pi i}{6}}$. 
        Therefore $c_1(P-\zeta_6^kP)+c_2(Q-\zeta_6^l Q)=0$. 
        Since $\mathrm{span}_{\Z[\omega]}\{P\} \cap \mathrm{span}_{\Z[\omega]}\{Q\}= \{ 0\}$, and $k \in \{1,2,4,5\}$, we have $c_1(1-\zeta_6^k) = 0$ and $c_1=0$. 
    \end{proof}

    Now we discuss case (iii) and case (iv).

    When case (iii) happens, cases C,D,E,F do not happen because $\sqrt{C} \in \Q(\omega)$ implies that $[\Q(\omega,\sqrt{AC}):\Q(\omega)]=[\Q(\omega,\sqrt{A}):\Q(\omega)]=2$. Similarly, when (iv) happens, cases C,D,E,F do not happen. 

    Cases A(iii) and A(iv). When case A happens, both $\alpha$ and $\beta$ have degree $3$ over $\Q(\omega,\sqrt{AC})$. By the Claim, we know that Statement (\ref{rk-1-equiv}) does not hold.

    Cases B(iii) and B(iv). We will prove that Statement (\ref{rk-1-equiv}) holds. Since $\alpha \in \Q(\omega,\sqrt{AC})$, let $\alpha=a+b\sqrt{AC}$ where $a,b\in \Q(\omega)$. Since $\sqrt{AC} \notin \Q(\omega)$, by the proof of (\ref{(2)}), we have $ \beta=\omega^i(a-b\sqrt{AC})$ for some $i \in \{0,1,2\}$. 
    Let $Q'=\omega^{-i} Q = \left(\left(a-b\sqrt{AC}\right)t,\ \sqrt{A}t^2-\sqrt{C}t\right)$. 
    Recall that $P= \left(\left(a+b\sqrt{AC}\right)t,\ \sqrt{A}t^2+\sqrt{C}t\right)$. 

    Case B(iii). We prove that $P-Q'$ is fixed by all $\sigma \in G$. 
    Let $\sigma \in G$. 
    If $\sigma$ fixes $\sqrt{A}$, since $\sqrt{C} \in \Q(\omega)$, we have $\sigma(P)=P$ and $\sigma(Q')=Q'$. If $\sigma$ does not fix $\sqrt{A}$, then we have 
    $\sigma(P)=-Q'$ and $\sigma(Q')=-P$. 
    In any case, $P-Q'$ is fixed by $\sigma \in G$. 

    Case B(iv). We prove that $P+Q'$ is fixed by all $\sigma \in G$.
    Let $\sigma \in G$.
    If $\sigma$ fixes $\sqrt{C}$, since $\sqrt{A} \in \Q(\omega)$, we have $\sigma(P)=P$ and $\sigma(Q')=Q'$. If $\sigma$ does not fix $\sqrt{C}$, we have $\sigma(P)=Q'$ and $\sigma(Q')=P$. 
    In any case, $P+Q'$ is fixed by $\sigma \in G$.

    Therefore, Statement (\ref{rk-1-equiv}) holds for the pair $(c_1,c_2)=(1,-\omega^{-i})$ in case B(iii) and for the pair $(c_1,c_2)=(1,\omega^{-i})$ in case B(iv). 
    
    This completes our discussion for case (iii) and (iv). \\

    Now we discuss case (v). Now both case A and case B do not happen. We prove that Statement (\ref{rk-1-equiv}) holds in cases D(v), E(v), F(v), and does not hold in case C(v).

    In cases D(v) and F(v), we have $\beta \in \Q(\omega)$ and $\sqrt{A},\sqrt{C} \in \Q(\omega)$, so $Q \in M^G$ and the pair $(c_1, c_2)=(0,1)$ works for Statement (\ref{rk-1-equiv}). In case E(v), we have $\alpha \in \Q(\omega)$ and $\sqrt{A},\sqrt{C} \in \Q(\omega)$, so $P \in M^G$ and the pair $(c_1, c_2)=(1,0)$ works for Statement (\ref{rk-1-equiv}). 
    
    In Case C(v). Since both $\Q(\omega,\alpha)$ and $\Q(\omega,\beta)$ have degree $3$ over $\Q(\omega,\sqrt{AC})$, by the Claim, we know that $c_1=c_2=0$ and Statement (\ref{rk-1-equiv}) does not hold.\\

    In summary, Statement (\ref{rk-1-equiv}) holds if and only if we are in one of the cases B(iii), B(iv), D(v), E(v), F(v). Note that case F(v) is equivalent to the conditions in part (a). Therefore we complete our proof of part (b).

    (c) This is the only remaining case, so it happens if and only if both (a) and (b) do not happen. This completes the proof of Theorem ~\ref{result-(4,3,2)}.
\end{proof}

Theorem ~\ref{result-(4,3,2)} completes the proof of Subsection ~\ref{subsection-non-degen} and therefore Theorem ~\ref{result-(6,3,0)}.
\end{proof}

\section{Proof of Theorem~\ref{arith-gen-(6,3,0)}} \label{section-arith-gen}

\begin{proof}
    Let $\langle\tau\rangle=\Gal(\Q(\omega)/\Q)$, then $\tau(1+2\omega)=\tau(\sqrt{-3})=-\sqrt{-3}=-(1+2\omega)$.

    (a) This follows from Lemma~\ref{CMT-2.1} and Lemma~\ref{reduction}.

    (b) When $\rk_{\Z}E_1(\Q(t))=1$, by Theorem~\ref{result-(6,3,0)}, one can easily verify that $P \in E_1(\Q(\omega)(t))$. The fact that $P$ is not a torsion point follows either from \cite{Br} Theorem 1.5 or from the calculation of the height pairing in Remark~\ref{alt-pf-(2,1,0)}. 
    Since the cube root lies in $\Q$ and $\sqrt{A}\in \Q(\omega)$, we know that $\tau(P)=\begin{cases}
        P \quad \text{if $\sqrt{A} \in \Q$}\\
        -P \quad \text{otherwise}
    \end{cases}$.
    This proves part (b).

    (c) when $\rk_{\Z}E_2(\Q(t))>0$, one can easily verify that $P \in E_2(\Q(\omega)(t))$. In case of $A=0$, the fact that $P$ is non-torsion follows either from Theorem 1.3 in \cite{Br}, or from Remark~\ref{alt-pr-(3,2)}. In case of $A \neq 0, C=0$, a change of variables in Lemma~\ref{basics}(c) changes this problem to the previous case. In case of $A,C \neq 0, B^2-4AC=0$, the point $P$ was constructed by following Subsection~\ref{subsection-degen} part (iii) and is therefore non-torsion.

    Note that when $B^2-4AC=0$, we have $\sqrt{A} \in \Q \iff \sqrt{C} \in \Q$. Note also that the cube roots lie in $\Q$. Therefore, $\tau(P)=\begin{cases}
        P \quad \text{if $\sqrt{A},\sqrt{C} \in \Q$}\\
        -P \quad \text{otherwise}
    \end{cases}$ in all cases. 
    This proves part (c).

    (e) By the change of variables in Lemma~\ref{basics}(c) and using Theorem~\ref{arith-gen-(6,3,0)} part (b), we get part (e).

    (d)(i) When $\sqrt{A},\sqrt{C} \notin \Q$, the cube roots are real, and $\begin{cases}
        \tau(P)=-P,\\
        \tau(Q)=-Q,
    \end{cases}$ as desired.

    When $\sqrt{A},\sqrt{C} \in \Q$, the cube roots are real, and $\begin{cases}
        \tau(P)=P,\\
        \tau(Q)=Q,
    \end{cases}$ as desired.

    When $\begin{cases}
        \sqrt{A} \notin \Q,\\
        \sqrt{C} \in \Q,
    \end{cases}$ since $B \in \Q$, we know that the two cube roots are not real and are conjugate in $\Q(\omega)/\Q$. 
    Therefore $\begin{cases}
        \tau(P)=-Q,\\
        \tau(Q)=-P,
    \end{cases}$ Hence $\begin{cases}
        \tau(P-Q)=P-Q,\\
        \tau(P+Q)=-(P+Q),
    \end{cases}$ as desired.

    When $\begin{cases}
        \sqrt{C} \notin \Q,\\
        \sqrt{A} \in \Q,
    \end{cases}$ the two cube roots are non-real and conjugates. Now we have $\begin{cases}
        \tau(P)=Q,\\
        \tau(Q)=P,
    \end{cases}$ so $\begin{cases}
        \tau(P+Q)=P+Q,\\
        \tau(P-Q)=-(P-Q),
    \end{cases}$ as desired.

    (d)(ii) As we have taken the two cube roots to be conjugate in $\Q(\omega,\sqrt{AC})/\Q(\omega)$, which is possible by Theorem~\ref{result-(4,3,2)} case B(iii),B(iv), let 
    $$P=\Big( \big(\alpha+\beta\sqrt{AC}\big)\cdot t, \sqrt{A}t^2+\sqrt{C}t \Big), \quad Q=\Big( \big(\alpha-\beta\sqrt{AC}\big)\cdot t, \sqrt{A}t^2-\sqrt{C}t \Big)$$
    for some $\alpha,\beta \in \Q(\omega)$ and let $\langle \sigma \rangle = \Gal(\Q(\omega,\sqrt{AC})/\Q(\omega))$.
    Therefore, \\
    when $\begin{cases}
        \sqrt{A} \notin \Q(\omega)\\
        \sqrt{C} \in \Q(\omega)
    \end{cases}$we have $\begin{cases}
        \sigma(P)=-Q\\
        \sigma(Q)=-P
    \end{cases}$; and when $\begin{cases}
        \sqrt{A} \in \Q(\omega)\\
        \sqrt{C} \notin \Q(\omega)
    \end{cases}$we have $\begin{cases}
        \sigma(P)=Q\\
        \sigma(Q)=P.
    \end{cases}$
    Hence the $R$ we constructed is a generator for $E_2(\Q(\omega)(t))$ as a $\Z[\omega]$-module up to finite index.

    When $\tau(R)=-R$, we know that $\{(1+2\omega)R\}$ is a basis (up to finite index) for the group $E_2(\Q(t))$. 
    
    When $\tau(R) \neq -R$, since $\tau^2=id$, we know that $R+\tau(R) \in E_2(\Q(t))\backslash\{0\}$. 
    Since $A,C,B^2-4AC \neq 0$ in part (d) of Theorem~\ref{arith-gen-(6,3,0)}, we know that 
    the $\Z[\omega]$-module $E_2(\bQ(t))$ is a free-module by Lemma~\ref{lem-geo-rk-2}. Hence $R+\tau(R)$ is not a torsion element in the $\Z[\omega]$-module and therefore generates the rank-one group $E_2(\Q(t))$ up to finite index.

    (d)(iii) Recall that $R$ is the one of $P,Q$ that lies in $E_2(\Q(\omega)(t))$, which is non-torsion.

    When $\tau(R)=-R$, we know that $\{(1+2\omega)R\}$ is a basis (up to finite index) for the group $E_2(\Q(t))$.

    When $\tau(R) \neq -R$, 
    by the same proof as in part (d)(ii), we know that $\{R+\tau(R)\}$ is a basis for $E_2(\Q(t))$ up to finite index.

    This completes the proof of Theorem~\ref{arith-gen-(6,3,0)}.
\end{proof}

\section{Largest rank, examples, and further remarks} \label{section-remarks}

In this section, we find the largest possible rank of $E_0(\Q(t))$, give examples to achieve this rank in different ways, and give an alternative proof for the parts in \cite{Br} that we have used.

\begin{prop}
    We have $\rk_{\Z}E_0(\Q(t))\leq 3$.
\end{prop}
\begin{proof}
    Suppose that $\rk_{\Z}E_0(\Q(t)) \geq 4$, then we have $\sqrt{A},\sqrt{C}\in \Q(\omega)$, and all of $B+2\sqrt{A}\sqrt{C}$, $B-2\sqrt{A}\sqrt{C}$, $4A$ and $4C$ are cubes in $\Q(\omega)$. So we have $16AC$ is a cube in $\Q(\omega)$. To see that $4\sqrt{A}\sqrt{C}$ is a cube in $\Q(\omega)$, let $\theta=\sqrt[3]{4\sqrt{A}\sqrt{C}}$, then $\theta^3=4\sqrt{A}\sqrt{C} \in \Q(\omega)$, and $\theta^2=\sqrt[3]{16AC} \in \Q(\omega)$, so $\theta \in \Q(\omega)$.
    
    Now since $B+2\sqrt{A}\sqrt{C}$, $B-2\sqrt{A}\sqrt{C}$ and $4\sqrt{A}\sqrt{C}$ are all cubes in $\Q(\omega)$, let them be $\alpha^3$, $\beta^3$ and $\gamma^3$, where $\alpha,\beta,\gamma\in \Q(\omega)=\Q(\sqrt{-3})$. Then we have $\alpha^3-\beta^3=\gamma^3$. 
    However, a classical result (see \cite{IR} Proposition 17.8.1) says that
    this is impossible unless at least one of $\alpha,\beta,\gamma$ is zero, which implies either $AC=0$ or $B^2-4AC=0$ and never gives $\rk_{\Z}E_2(\Q(t))=2$. Therefore, it is not possible to achieve $\rk_{\Z}E_0(\Q(t))=4$.
\end{proof}

Now we give some examples to achieve $\rk_{\Z}E_0(\Q(t))=3$ in different ways.

\begin{example}
    Now we give an example when $\rk_{\Z}E_0(\Q(t))=3$ is achieved from  $\rk_{\Z}E_1(\Q(t))=\rk_{\Z}E_2(\Q(t))=\rk_{\Z}E_3(\Q(t))=1$.
    
    Let $n \geq 2$ be an integer and
    $$A=C=(4n^3-4)^2, \quad B=(4n^3-2)A,$$
    we have $\sqrt{A}=\sqrt{C}\in \Q\backslash\{0\}$ and $B^2-4AC \neq 0$. We also have $-(2\sqrt{A}\sqrt{C}-B) = B-2A=(4n^3-4)^3$ is a cube in $\Q(\omega)\backslash\{0\}$, and $-(-2\sqrt{A}\sqrt{C}-B)=B+2A=4n^3(4n^3-4)^2=(4n)^3 (n^3-1)^2$ is not a cube in $\Q(\omega)$ for $n \geq 2$, so $\rk_{\Z}E_2(\Q(t))=1$.
    
    We also have $\frac{B^2-4AC}{4C}=\frac{B^2-4AC}{4A} = (B-2A) \cdot \frac{B+2A}{4A}=(4n^3-4)^3 \cdot n^3$ is a cube in $\Q(\omega)$, and recall that $\sqrt{A}=\sqrt{C}\in \Q$, so $\rk_{\Z}E_1(\Q(t))=\rk_{\Z}E_3(\Q(t))=1$. \qed 
\end{example}

\begin{example}
    Now we give an example when $\rk_{\Z}E_0(\Q(t))=3$ is achieved from $\rk_{\Z} E_1(\Q(t))=1$, $\rk_{\Z}E_2(\Q(t))=2$ and $\rk_{\Z} E_3(\Q(t))=0$ (or the symmetric counter-part). Here we give an example when $\sqrt[3]{2\sqrt{A}\sqrt{C}-B}, \allowbreak \sqrt[3]{-2\sqrt{A}\sqrt{C}-B} \in \Q(\omega)\backslash\Q$.
    
    Let  
    $$A=16, \quad B=280, \quad C=-972,$$ we have $\sqrt{A}=4 \in \Q(\omega)$, $\sqrt{C}=18\sqrt{-3} \in \Q(\omega)$. Now we have $\sqrt{AC}=72\sqrt{-3}$ and $\sqrt[3]{\pm 2\sqrt{A}\sqrt{C}-B}
    =\sqrt[3]{-280 \pm 144 \sqrt{-3}}
    =2(1\mp 2\sqrt{-3})$. Therefore, $\rk_{\Z}E_2(\Q(t))=2$.

    Since $\sqrt[3]{\pm 2\sqrt{A}\sqrt{C}-B} \in \Q(\omega)\backslash\{0\}$, their product $\sqrt[3]{B^2-4AC}$ lies in $\Q(\omega)\backslash\{0\}$. Since $4A=64=4^3$, we have $\rk_{\Z}E_1(\Q(t))=1$. Since $4C=-3888=-2^4 \cdot 3^5$ is not a cube in $\Q(\omega)$, we have $\rk_{\Z}E_3(\Q(t))=0$. \qed
\end{example}

Now we give an alternative proof for the results in \cite{Br} that we have used.

\begin{remark} \label{alt-pf-(2,1,0)}
    Recall that $E_1: y^2=x^3+(At^2+Bt+C)$. When $A=0$ or $B^2-4AC=0$, computing the singular fibres and using the Shioda-Tate formula (\cite{SS} Section 7.2) gives $\rk_{\Z} E_1(\bQ(t))=0$. 
    Now we suppose that $A \neq 0$ and $B^2-4AC \neq 0$. 
    Then the only singular fibre not of type II is at $t=\infty$, which has type $IV^*$, and Shioda-Tate's formula gives us $\rk_{\Z[\omega]}E_1(\bQ(t))=1$. 
    We observe that 
    $P=\left( \sqrt[3]{\frac{B^2-4AC}{4A}}, \sqrt{A}t +\frac{B}{2\sqrt{A}}   \right)$ is a point on $E_1(\bQ(t))$. On the $s$-chart, the surface is  $E_1: y^2=x^3+(Cs^6+Bs^5+As^4)$ and the section is $P=\left( \sqrt[3]{\frac{B^2-4AC}{4A}} s^2, \sqrt{A}s^2 +\frac{B}{2\sqrt{A}}s^3   \right)$. We have $(P \cdot O)=0$ and the section $P$ after the blow-ups intersects a non-identity component because $P$ intersects the singularity at $s=0$ before the blow-ups. Therefore $\langle P,P \rangle=2+0-\frac{4}{3}=\frac{2}{3}$. So this $P$ is not a torsion point, and Lemma~\ref{rank-criteria} (or \cite{Kl} Lemma 2.8) gives an alternative proof of Proposition ~\ref{result-(2,1,0)}. 
\end{remark}

\begin{remark} \label{alt-pr-(3,2)}
    For $E: y^2=x^3+(At^3+Bt^2)$ where $A \neq 0$, we observe that 
    $P=(-\sqrt[3]{A}t, \sqrt{B}t)$ is a point in $E(\bQ(t))$. 
    When $B=0$, we know by the Shioda-Tate formula (\cite{SS} Section 7.2) that $\rk_{\Z}E(\bQ(t))=0$. Now we suppose (additionally) that $B \neq 0$, then Shioda-Tate's formula gives us $\rk_{\Z[\omega]} E(\bQ(t))=1$. 
    The singular fibre at $t=0$ is of type IV, at $t=\infty$ is of type $I_0^*$, at $t=-\frac{B}{A}$ is of type II, and there are no other singular fibres.
    In the $s$-chart, the surface is $E: y^2=x^3+(Bs^4+As^3)$, the section is $P=(-\sqrt[3]{A}s, \sqrt{B}s^2)$. Therefore, we have $(P \cdot O)=0$, and $P$ intersects a non-identity component at both $t=0$ and $t=\infty$ after blow-ups because $P$ intersects the singularities of those fibres before the blow-ups. Therefore $\langle P,P \rangle=2+0-(\frac{2}{3}+1)=\frac{1}{3}$, and $P$ is not a torsion point. By Lemma~\ref{rank-criteria} (or \cite{Kl} Lemma 2.8), we get an alternative proof of Proposition ~\ref{result-(3,2)}.
    
\end{remark}

\end{document}

%% file: Galois-diagrams.tex
\newcommand{\CubicDiagram}[5][0.24]{\begin{subfigure}{#1\textwidth}
\centering
\begin{tikzcd}[row sep=0.5cm, column sep=-1cm, ampersand replacement=\&]
    \Q(\omega,\alpha) \arrow[dr, dash, "#3"']
    \& \& \Q(\omega,\beta) \arrow[dl, dash, "#4"]\\
    \& \Q(\omega,\sqrt{AC}) \arrow[d, dash, "#5", end anchor=north] \& \\
    \& \Q(\omega) \&
\end{tikzcd}
\caption*{Case #2}
\end{subfigure}}

\newcommand{\QuadraticDiagram}[6][0.3]{\begin{subfigure}{#1\textwidth}
\centering
\begin{tikzcd}[row sep=0.5cm, column sep=-1.3cm, ampersand replacement=\&]
    \& \Q(\omega,\sqrt{A},\sqrt{C})
      \arrow[dl, dash, "#3"', end anchor=north]
      \arrow[dr, dash, "#4", end anchor=north]
    \& \\
    \Q(\omega,\sqrt{A}) \& \& \Q(\omega,\sqrt{C})\\
    \& \Q(\omega)
      \arrow[ul, dash, "#5", end anchor=south]
      \arrow[ur, dash, "#6"', end anchor=south]
    \&
\end{tikzcd}
\caption*{Case (#2)}
\end{subfigure}}